\documentclass{amsart}
\usepackage{amssymb}
\usepackage{color}  
\definecolor{darkgreen}{cmyk}{1,0,1,.2}
\definecolor{m}{rgb}{1,0.1,1}
\definecolor{green}{cmyk}{1,0,1,0}

\definecolor{test}{rgb}{1,0,0}
\definecolor{cmyk}{cmyk}{0,1,1,0}

\long\def\red#1{\textcolor {red}{#1}}  
\long\def\blue#1{\textcolor {blue}{#1}}

\newtheorem{Equation}{}[section]

\newtheorem{example}[Equation]{Example}
\newtheorem{theorem}[Equation]{Theorem}
\newtheorem{proposition}[Equation]{Proposition}
\newtheorem{lemma}[Equation]{Lemma}

\newtheorem{definition}[Equation]{Definition}

\newtheorem{remark}[Equation]{Remark}

\def\ch{\operatorname{ch}}

\def\I{\operatorname{I}}

\def\End{\operatorname{End}}

\def\wexp{\what{\exp}}
\def\whexp{\what{\exp}}

\def\Supp{\operatorname{Supp}}

\def\Tr{\operatorname{Tr}}
\def\TR{\operatorname{TR}}
\def\tr{\operatorname{tr}}

\def\C{\mathbb C}

\def\N{\mathbb N}
\def\R{\mathbb R}
\def\S{\mathbb S}
\def\Z{\mathbb Z}

\def\maD{{\mathcal D}}

\def\maF{{\mathcal F}}
\def\maM{{\mathcal M}}
\def\cG{{\mathcal G}}

\def\maG{{\mathcal G}}
\def\cG{{\mathcal G}}
\def\maH{{\mathcal H}}

\def\maN{{\mathcal N}}

\def\cS{{\mathcal S}}
\def\maS{{\mathcal S}}

\def\what{\widehat}
\def\wtit{\widetilde}

\def\hM{\what{M}}

\def\pit{\pitchfork}

\def\wL{\widetilde{L}}

\def\dd{\displaystyle}
\def\pa{\partial}

\def\ep{\epsilon}

\begin{document}

%%%%%%%%%%%%%%%%%%%%%%%%%%%%%%%%%%%%%%%%%%%%%%%%%%%%%%%
%                                                     %
%   THE MANUSCRIPT BEGINS HERE                        %
%                                                     %
%%%%%%%%%%%%%%%%%%%%%%%%%%%%%%%%%%%%%%%%%%%%%%%%%%%%%%%

%%% TITLE

\title[Atiyah covering index theorem for Riemannian foliations \today]
{Atiyah covering index theorem for Riemannian foliations\\
\today}

%%% AND THE AUTHORS ARE:

\author[M-T. Benameur]{Moulay-Tahar Benameur}
\address{Institut Montpellierain Alexander Grothendieck, UMR 5149 du CNRS, Universit\'e de Montpellier}
\email{moulay.benameur@umontpellier.fr}

\author[J.  L.  Heitsch \today]{James L.  Heitsch}
\address{Mathematics, Statistics, and Computer Science, University of Illinois at Chicago} 
\email{heitsch@uic.edu}


\begin{abstract} We use the symbol calculus for\ foliations developed in \cite{BH2016a}  to  derive a cohomological formula for the Connes-Chern character of the Type II spectral triple given in \cite{BH2016b}.  The same proof works for the Type I spectral triple of Connes-Moscovici. The cohomology classes of the two Connes-Chern characters induce the same map on the image of the maximal Baum-Connes map in K-theory, thereby proving an Atiyah $L^2$ covering index theorem.
\end{abstract}

\maketitle
\tableofcontents

\section{Introduction}

In \cite{CM95}, Connes and Moscovici developed  ``a general, in some sense universal, local index formula for arbitrary spectral triples of finite summability degree, in terms of the Dixmier trace and its residue-type extension."  Because of its extremely wide applicability, their formulas are quite complicated.  See also \cite{CM98}.   In \cite{K97}, Kordyukov considered a Type I Connes-Moscovici spectral triple naturally associated to a compact foliated manifold with a transversely elliptic operator with holonomy invariant principal symbol.  He proved that the spectral triple is finite dimensional,  and that its spectrum is simple and contained in the set $\{m \in \N \,\, | \,\, m \leq q \}$, where $q$ is the co-dimension of the foliation.  In \cite{BH2016b}, we extend Kordyukov's result to certain non-compact manifolds, including non-compact Galois coverings of compact foliated manifolds, which yields a Type II spectral triple (also called a semi-finite spectral triple), see \cite{B03, BF06},  with the same properties as the Type I Connes-Moscovici spectral triple.  

In this paper we use the symbol calculus for foliations developed in \cite{BH2016a} to derive a cohomological formula for the Connes-Chern characters of these Type I and Type II spectral triples.  In particular, we assume that we have a Riemannian foliation $F$ of a compact manifold $M$ with a transverse Dirac operator, and that there is a complementary foliation transverse to $F$.  The formulas we obtain are similar to that for the classical case of a Dirac operator on a compact manifold given in \cite{BF}. The same proof works for both the Type I spectral triple of Connes-Moscovici and Kordyukov and the Type II spectral triple of \cite{BH2016b}.  The Connes-Chern characters we obtain induce the same map on the image of the Baum-Connes map in K-theory, \cite{BC00}, so we obtain an Atiyah $L^2$ type covering index theorem for compact foliated manifolds with transversely elliptic operators.

It is possible to dispense with the assumption that $M$ admits a transverse foliation.  We show how to do this in Section \ref{statementCM} using the standard Morita reduction to transversals, and we consider spectral triples associated to a chosen complete transversal of $F$.   There are several reasons why we prefer to work on $M$.  Working on $M$ doesn't involve choices while working on $T$ does, so working on $M$  is more natural. Moreover, for  important generalizations, see \cite{BH2016b}, we want a  formula  in terms of characteristic classes on the ambiant manifold $M$.   There are natural examples where the restriction to a transversal doesn't work, but our techniques here do.  See Example \ref{Notrans}.

\medskip\noindent
{\em Acknowledgements.} It is a pleasure to thank A. Carey, P. Carrillo-Rouse, T. Fack, G. Hector, and P. Piazza
for helpful discussions.  We are also indebted to the referee for his cogent suggestions.
MB wishes to thank the french National Research Agency for support via the project ANR-14-CE25-0012-01 (SINGSTAR).

\section{Background}

We assume that the reader is familiar with the paper \cite{BH2016a}, and we  will freely use the notations of that paper.  In particular,  $F$ is a smooth Riemannian foliation  of the smooth closed Riemannian manifold $M$.  Then $(M,F)$ is of bounded geometry, so  all the leaves of $F$ and all the bundles associated to $M$ and $F$ are of bounded geometry.   Denote by $TM$ and $T^*M$  the tangent and cotangent bundles of $M$.   The dimension of $F$ is $p$ and the dimension of  $M$ is $n$, so the codimension of $F$ is $q=n-p$, which we assume is even.   If $q$ is not even, we replace $M$ by $M \times \S^1$ and $F$ by the obvious $p$ dimensional foliation it determines on $M \times \S^1$. The normal bundle of $F$ is denoted $\nu$ ($= TF^{\perp}$), and its conormal bundle is $\nu^*$.  The induced metrics on $\nu$ and $\nu^*$ are assumed to be bundle like.  The leaf of $F$ through a point $x$ is denoted $L_x$.    The  homotopy groupoid of $F$ is $\cG$, with the source and target maps $s, r: \cG \to M$.  $F_s$ is the foliation of $\cG$ whose leaves are given by $\wL_x = s^{-1}(x)$,  which is also denoted $\cG_x$.  $F_r$ is the foliation of $\cG$ whose leaves are given by $\wL^y = r^{-1}(y)$,  which is also denoted $\cG^y$.   We denote $\cG_x \cap \cG^y$ by $\cG_x^y$.   Note that  $r:\wL_x \to L_x$ is the simply connected covering map.   If $\gamma \in \cG_x$, the holonomy along $\gamma$  from $x = s(\gamma)$ to $r(\gamma)$ in $L_x$ is denoted $h_{\gamma}$. 

Denote by $E \to M$ a smooth complex vector bundle over $M$ with a Hermitian structure.  We assume that  $E$ is  basic, which means that there is an action on it by the {\em holonomy groupoid} and hence also by the homotopy groupoid, and that the Hermitian connection on it, denoted $\nabla^E$, is locally projectable, so its curvature $\Omega^E$ is basic, that is locally a pull-back from a transversal.  

Assume that $F$ is transversely spin.  Fix a spin structure on  $\nu^*$, and denote by $\cS_{\nu} = \cS^+_{\nu} \oplus \cS^-_{\nu}$ the associated spin bundle, with its natural splitting.  The bundles $\cS^\pm_\nu$ are  automatically basic vector bundles over $M$.  The Hilbert space of $L^2$-sections of the bundle $\cS_{\nu} \otimes E$ over $M$ is denoted $\maH$, which is  $\Z_2$ graded  since  $\maS_{\nu}\otimes E$ is $\Z_2$-graded. 

The Levi-Civita connection on $\nu^*$ (respectively $\nu$) is  denoted $\nabla^{\nu}$, which is locally the pull-back of the Levi-Civita connection on a transversal.  It is also known as a Bott or basic connection. The connection  $\nabla^{\nu}$ induces a connection on $\cS_{\nu}$, which when combined with the connection $\nabla^E$ on $E$ gives a connection, also denoted $\nabla^{\nu}$, on $\cS_{\nu} \otimes E$.   The context should make clear on which bundle $\nabla^{\nu}$ is acting.

The  transverse Dirac operator $D$ on $\maH$  is given as follows.  Choose a local orthonormal basis $f_1,...,f_q$   of $\nu^*$ with dual orthonormal basis  $e_1,...,e_q$ of $\nu$.  For $u \in \maH$, set 
$$
\wtit{D}(u) \,\, = \,\, \sum_{1\leq i\leq q} f_i \cdot \nabla^{\nu}_{e_i}u,
$$
where $f_i \cdot$  is the operator $c(f_i)$, Clifford multiplication by $f_i$.  In general, $\wtit{D}$ is not self-adjoint.  The mean curvature vector field of $F$   is 
$\mu= \sum_{j=1}^p p_{\nu}(\nabla_{X_j} X_j)$ where $X_1,...,X_p$ is a local orthonormal framing of $TF$, $\nabla$ is the Levi-Civita connection on $M$, and $p_{\nu}:TM \to \nu$ is the projection.  When we think of $\mu$ as a covector (the isomorphism $\nu \simeq \nu^*$ being given by the inner product), then we denote Clifford multiplication by it by $c(\mu)$, and it is given by, see \cite{BH2016a},
$$
c(\mu) \,\, = \,\,  \sum_{j=1}^p  \sum_{i=1}^q \langle  [e_i,X_j],X_j \rangle f_i.
$$
The transverse Dirac operator $D$ associated to $F$ is
$$
D \,\, = \,\,  \wtit{D} \,\, - \,\, \frac{1}{2}c(\mu),
$$
which is self-adjoint.  See \cite{K07,GlK91}.

\section{Statement of the main result} \label{statement}

Recall  the type I spectral triple introduced in \cite{CM95} and \cite{K97}:

\begin{definition}
The Connes-Moscovici Type I spectral triple associated to $F$ is $(C_c^\infty (\cG), \maH, D)$.  The usual operator trace yields the super-trace, denoted $\Tr_s$, constructed using the Berezin integral \cite{Getz, BGV92}, and $a \in C^{\infty}_c(\cG)$  acts on $\maH$ by 
$$
a(u)(x) \,\, = \,\, \int_{\cG_{x}} a(\gamma^{-1}) h_{\gamma^{-1}}(u(r(\gamma))d\gamma.
$$
\end{definition}

The main theorem of \cite{K97} is (see also \cite{CM95}),
\begin{theorem} 
$(C_c^\infty (\cG), \maH, D)$ is a regular even spectral triple with simple dimension spectrum contained in the set  $\{k \in \N \, | \, k \leq q \}$.
\end{theorem}

{{We now recall an important special case of the main result of \cite{BH2016b},  which will be used in the sequel.    In \cite{BH2016b}, we were concerned with a general bounded geometry foliation $(\hM, \what\maF)$ endowed with a proper action of a discrete group $\Gamma$, and with the associated transverse noncommutative geometry as provided by Connes' formalism of regular spectral triples, \cite{ConnesNCG}, and their semi-finite version as introduced in \cite{BF06}. By using a classical reduction method, we were led to the corresponding problem for the class of {\em{Riemannian bifoliations}} $(\hM, \what\maF\subset \what\maF')$, see again \cite{BH2016b}. }}

{{In the case of a free and proper action, we get a bounded geometry $\Gamma$-covering $(\hM, \what\maF\subset \what\maF')\to (M, \maF\subset \maF')$ of bounded geometry bifoliations.  The main result of \cite{BH2016b} is the following. }}
{{\begin{theorem} 
 Assume that ${\what D}$ is a transversely elliptic $\Gamma$-invariant pseudodifferential operator in the Connes-Moscovici sense for the larger foliation $\what\maF'$, which is essentially self-adjoint  with holonomy invariant transverse principal symbol. 
Then  the triple $(C_c^\infty (\cG), (\maM, \tau), \what{D})$ is a semi-finite spectral triple which is finitely summable of dimension equal to the Beals-Greiner codimension of $\what\maF'$.
\end{theorem}}}

{{Here 
${\what D}$ acts on the Hilbert space $\what\maH$ of $L^2$-sections of a bundle over $\hM$.   $\cG$ is  the monodromy groupoid of $\maF$,    $\maM= B(\what\maH)^\Gamma$ is the Atiyah von Neumann algebra of $\Gamma$-invariant operators on $\what\maH$, with its semi-finite trace $\tau$ as defined for instance in \cite{AtiyahCovering}. }}

{{For the special case of  a $\Gamma$-covering $(\hM, \what\maF)\to (M, \maF$), (so $\what\maF' = \what\maF$), of a Riemannian foliation, this theorem becomes
\begin{theorem}
Assume that ${\what D}$ is a transversely elliptic $\Gamma$-invariant pseudodifferential operator on $\hM$, which is essentially self-adjoint  with holonomy invariant transverse principal symbol. 
Then  the triple $(C_c^\infty (\cG), (\maM, \tau), \what{D})$ is a type II spectral triple which is finitely summable of dimension equal to the codimension of the foliation,  with simple semi-finite dimension spectrum contained in the set  $\{k \in \N \, | \, k \leq q \}$.
\end{theorem}
}}

For technical reasons which will be explained later, we now assume that the normal bundle $\nu$ is integrable, that is $M$ has a $q$ dimensional foliation which is transverse to $F$.   This foliation is denoted $F_{\pit}$.   Since $F$ is transversely spin,  $F_{\pit}$ is a spin foliation.

Now suppose again that  $\rho: \what{M} \to M$ is a Galois covering, and  consider $M \subset \what{M}$ as a fundamental domain.  Set $\what{F} = \rho^{-1}F$,  with leaves $\what{L}$,  $\what{\nu} = \rho^*\nu$, and $\what{\nu}^* = \rho^*\nu^*$.   We also have the corresponding pulled back objects $\what{\maS}_{\nu} \otimes \what{E}$, $\what{\maH} = L^2(\what{M}, \what{\cS}_{\nu} \otimes \what{E})$, $\what{D}$, $\what{\cG}$, etc.  

Note that in general the inverse image of a leaf $L$ of $F$ may consist of a number of leaves of $\what{F}$, and that for each of those leaves $\rho:\what{L} \to L$ is a connected covering of $L$.
Since the simply connected covering $\wtit{\what{L}} \to \what{L}$ of the covering $\what{L} \to L$ is just the simply connected covering  $\wtit{L} \to L$, the following  is obvious.
\begin{lemma}
Given $\gamma \in \cG_{\rho(\what{x})}$, there exists a unique $\what{\gamma} \in \what{\cG}_{\what{x}}$ so that 
$\rho \circ \what{\gamma}  = \gamma$.
\end{lemma}
This lemma is not true in general if we use holonomy groupoids instead of homotopy groupoids.
 
\begin{definition}\label{action}Let  $a \in C^{\infty}_c(\cG)$  act on $\what{u} \in \what{\maH}$ by 
$$
a(\what{u})(\what{x}) \,\, = \,\, \int_{\cG_{\rho(\what{x})}} a(\gamma^{-1}) h_{\what{\gamma}^{-1}}(\what{u}(\what{r}(\what{\gamma})))d\gamma  \,\, = \,\, 
\int_{\what{\cG}_{\what{x}}} a((\rho \circ \what{\gamma})^{-1}) h_{\what{\gamma}^{-1}}(\what{u}(\what{r}(\what{\gamma})))d\what{\gamma}.
$$
\end{definition}

\begin{definition}
The  {{Atiyah-Connes}} spectral triple associated to a (possibly non-compact) Galois foliation cover $\rho:(\what{M}, \what{F})\to (M,F)$, with covering group $\Gamma$ as above, is $(C^{\infty}_c(\cG),  B(\what{\maH})^{\Gamma},  \what{D})$.   The trace is the Atiyah-trace which yields the super trace, denoted $\tau_s$, constructed using the Berezin integral, but restricted to the fundamental domain $M \subset \what{M}$. 
\end{definition} 

Note the following.

1)  Since we are using the homotopy groupoid, we want the algebra of smooth functions on the homotopy groupoid  $\what{\maG}$ of $\what{F}$, which are $\Gamma$-invariant and $\Gamma$-compactly supported. This is precisely the algebra downstairs, that is $C_c^\infty(\cG)$.  

 2) The von Neumann algebra $B(\what{\maH})^{\Gamma}$ is the algebra used by  Atiyah in the $L^2$ covering index theorem.

 3) The Dirac operator $\what{D}$  on $\what{M}$ is defined with respect to $\Gamma$-invariant data and hence is $\Gamma$-invariant.  

4)  For a compact covering, this is a slight modification of the Type I Connes-Moscovici spectral triple above.

%5)  For a non-compact covering, we have, see \cite{BH2016b}, 
%\begin{theorem} 
% $(C^{\infty}_c(\cG),  B(\what{\maH})^{\Gamma}, \what{D})$ is a regular Type II even spectral triple with simple dimension spectrum contained in the set  $\{k \in \N \, | \, k \leq q \}$.
%\end{theorem}

Denote by $d_{\nu}$ the transverse de Rham differential for the foliated manifold $(\cG,F_s)$, \cite{T97}, and by $dx_F$ the volume form along the leaves of $F$.
Our main result is the following.

\begin{theorem}\label{CCchar}
Let $(\what{M}, \what{F})$ be a  covering foliation of $(M,F)$.  Assume that  $F$  has even codimension, is transversely spin, and that  its normal bundle $\nu$ is integrable.   Let $a_0, \ldots, a_k \in C_c^\infty (\cG)$.  The  sequences $(\phi_{k})_{k\geq 0}$, and $(\what{\phi}_{k})_{k\geq 0}$, $k$ even, where
$$
\phi_{k} (a_0, \ldots, a_k) \,\, = \,\,    \frac{1}{k!} \int_{M}  \sum_{\gamma \in \cG_x^x} (a_0  d_{\nu} a_1  \cdots d_{\nu} a_k)(\gamma^{-1})  \wedge \ch(E) \wedge \what{A} (\nu^*) \wedge  dx_F,
$$
and
$$
\what{\phi}_{k} (a_0, \ldots, a_k) \,\, = \,\,    \frac{1}{k!} \int_{M \subset \what{M}} 
 \sum_{\what{\gamma} \in \what{\cG}_x^x} (a_0  d_{\nu} a_1  \cdots d_{\nu} a_k)((\rho \circ\what{\gamma})^{-1})    
 \wedge \ch(\what{E}) \wedge {\what A} (\what{\nu}^*) \wedge  dx_F
$$
are cocycles in the $(b,B)$ bicomplex for $C_c^\infty (\cG)$.

The cohomology classes of  $(\phi_k)$ and $(\what{\phi}_k)$ are respectively the Connes-Chern characters of $(C_c^\infty (\cG), \maH,D)$ and   $(C_c^\infty (\cG), B(\what{\maH})^{\Gamma}, \what{D})$.  These classes induce the same map on the image of the maximal Baum-Connes assembly map in $K_*(C^*_{\max}(\cG))$, \cite{BC00}.  That is,  we have an Atiyah $L^2$ covering index theorem for foliation coverings.
\end{theorem}
{{\begin{remark}\label{BC-Free}
The maximal assembly map referred to in this paper is the map induced by leafwise index maps, see for instance \cite{ConnesBook}.  If $\cG$ is torsion-free, this map  is conjectured to be an isomorphism, see  \cite{BC00}.  When $\cG$ has torsion, it has a modified definition, and it is not an isomorphism in general.  See again \cite{BC00} and also \cite{TuHyper}. 
\end{remark}
}}

Note that the sums $\sum_{\cG_x^x}$ and $\sum_{\what{\cG}_x^x} $ are all finite since $a_0  d_{\nu} a_1  \cdots d_{\nu} a_k$ has compact support.
Note also that the  product in $a_0   d_{\nu} a_1 \cdots d_{\nu} a_k$ is the convolution wedge product on the groupoid, while the wedges are the simple pointwise wedge product.  Finally note that the periodic cyclic cohomology classes of $(\phi_k)$ and $(\what{\phi}_k)$ do not depend on either the metric on $M$ (so also on $\what{M}$), nor the choice of normal bundle $\nu$.   

\section{Symbols and the transverse symbol space}

To begin,  we construct a generalized exponential function $\whexp_{x}:TM_x \to M$.    Suppose $X = (X_{\nu}, X_F) \in \nu_x \oplus TF_x $.  The leaves  of $F$  and $F_{\pit}$ through $x$ are denoted $L_x$  and $L^{\pit}_x$, and their leafwise exponentials are denoted $\exp^F$ and $\exp^{\pit}$.
Denote parallel translation from $x$ to $\exp^F(X_F)$ along $\exp^F(tX_F)$, $t \in [0,1]$, in $L_x$  of the leafwise flat bundle $\nu$ by $h_{\exp^F(X_F)}$.  Set 
$$
\whexp(X) \,\, = \,\, \exp^{\pit}(h_{\exp^F(X_F)}(X_{\nu})).
$$
What we are doing is  exponentiating first in the leaf direction using the usual exponential function for $L_x$,  and then exponentiating in the transverse direction using  the usual exponential for $L^{\pit}_{\exp^F(X_F)}$.  

\begin{remark}\label{expremark}
The map $\whexp$ has the essential properties of the usual exponential function:  it is a local diffeomorphism from $0 \in TM_x$ to a neighborhood of $x$, whose differential at $0$ is the identity.   For $X \in \nu_x$, $\whexp(X)$ is a geodesic in the leaf of $F_{\pit}$ through $x$.

There are two natural foliations on $TM_x$, namely the planes parallel to $TF_x$ and $\nu_x$ respectively.  Since $F$ is Riemannian, $\whexp$ maps these foliations onto the foliations $F$ and $F_{\pit}$, respectively.

The results of \cite{BH2016a} still hold for the transverse symbol space $SC^{\ell}_{\pit} (M, E)$ defined below  if $\what{\exp}$ is used in place of the usual exponential function.  
\end{remark}

Choose a smooth bump function $\alpha$ on $M \times M$ which is  supported in a neighborhood of the diagonal, and equals one on a neighborhood of the diagonal.  We require that the support of $\alpha$ is close enough to the diagonal that $(\pi_M, \wexp)^{-1}:\Supp(\alpha) \to TM$ is a diffeomorphism onto the component of $(\pi_M, \wexp)^{-1}(\Supp(\alpha))$ which contains the zero section, where  $\pi_M:TM \to M$ is the projection. Suppose that $x' = \wexp(X)$ where $X \in TM_x$.  Given $u_x \in (\cS_{\nu} \otimes E)_x$, denote by $\mathcal{T}_{x,x'}(u_x)$ the parallel translation of $u_x$ along the radial line  $t \to \wexp(tX)$ from $x$ to $x'$.
\begin{definition}
Given an operator $P$ on sections of $\cS_{\nu} \otimes E$, its symbol $\varsigma (P)$ is defined as follows. Let $x \in M$, $\xi  \in  T^*M_x$, and $u_x \in (\cS_{\nu} \otimes E)_x$.    Set 
$$
\varsigma(P)(x,\xi)(u_x) \,\, = \,\,  
P \Big(x' \mapsto e^{i\langle \wexp^{-1}_x(x'), \xi  \rangle}
\alpha(x,x')\mathcal{T}_{x,x'}(u_x)\Big) \, |_{x' = x}.
$$
Similarly for $\what{P}$ acting on sections of $\what{\cS}_{\nu} \otimes \what{E}$.
\end{definition}
If we write $T^*M = \nu^* \oplus T^*F$, then $\xi \in T^*M$ may be written as $\xi = (\eta,\zeta)$, and we may also write $\varsigma(P)(x,\xi) = \varsigma(P)(x, \eta, \zeta)$.  
Denote by $\pi:\nu^* \to M$ the projection.

\begin{definition} \label{symbdef}
The symbol space $S^{\ell}_{\pit} (M, E)$ consists of all $p(x,\eta) \in  C^{\infty}(\nu^*, \pi^*(\End (\cS_{\nu} \otimes E)))$  so that, for any multi-indices $\alpha$ and $\beta$,  there is a constant $C_{\alpha,\beta} > 0$,  such that
$$
||  \,   \pa^{\alpha}_{\eta}   \pa^{\beta}_{x}p(x,\eta) \, || \,\, \leq \,\, 
C_{\alpha,\beta}(1 +  |\eta|)^{\ell- |\alpha|}.
$$
The topology on $S^{\ell}_{\pit} (M, E)$  given by the semi-norms  
$$
\rho_{\alpha,\beta}  \,\, = \,\  \inf  \big{\{}  C_{\alpha,\beta} \, | \,\,   ||  \,   \pa^{\alpha}_{\eta}   \pa^{\beta}_x 
p(x,\eta) \, || \,\, \leq \,\, 
C_{\alpha,\beta}(1 +  |\eta|)^{\ell- |\alpha|}  \big{\}} .
$$
\end{definition}

If we replace the variable $\eta$ with the variable $\sigma$,  this is a subspace of the space $S^{0,\ell}(M,E)$  of \cite{BH2016a}, for the \underline{transverse} \underline{foliation} $F_{\pit}$.   We shall consider it so, while retaining the use of the variable $\eta$.  

\medskip
Of course, $\pa^{\alpha}_{\eta}  \pa^{\beta}_{x}p(x,\eta)$ only makes sense if we specify local coordinates.  We will use the so-called ``normal coordinates" on $M$.  Normal coordinates at a point $x \in M$ are given by choosing a neighborhood $U_x$ of $0 \in TM_x$ on which $\wexp:TM_x \to M$ is a diffeomorphism, and orthonormal bases of $\nu_x$ and $TF_x$, which define coordinates $(x_1,\ldots,x_n)$ on $\nu_x \oplus TF_x = TM_x$.   This then defines coordinates (also denoted $(x_1,\ldots,x_n)$) in the neighborhood $\wexp(U_x)$ of $x$.  In addition, it also induces coordinates on $\nu^*_x$, $T^*F_x$, and $T^*M_x$.  

Denote by $\wedge^* \nu^*$ the complexified Grassmann algebra bundle. Then
$$
C^{\infty}(\nu^*, \pi^*(\End(\cS_{\nu} \otimes E)))  \,\, \cong \,\,   
C^{\infty}(\nu^*,  \pi^*(\wedge^* \nu^* \otimes  \End( E))),
$$
as $\wedge^* \nu^*  \cong {\rm Cliff}(\nu^*) $,  where ${\rm Cliff}(\nu^*)$ is the Clifford algebra.  Since $q$ is even,  $\End(\cS_{\nu})  \cong  {\rm Cliff}(\nu^*)$.  The reader should note carefully that when we represent endomorphisms of $\cS_{\nu}$ as elements of $C^{\infty}(\wedge^* \nu^*)$ and we compose them, the operation we use is Clifford multiplication and not wedge product.  Note also that we use the convention $\omega \cdot \omega =  -\langle \omega, \omega \rangle =-||\omega||^2$ for co-vectors. 

Following \cite{Getz}, we treat Clifford multiplication by a {\em normal} $k$-co-vector as a differential operator of order $k$.  Thus if  $p \in  C^{\infty}( \nu^*, \pi^*(\wedge^k \nu^* \otimes \End (E))) \cap S^{\ell-k}_{\pit}(M, E)$, we say $p$ has grading $\ell$.  
\begin{definition}
The  transverse  symbol space $SC^{\ell}_{\pit} (M, E)$ is 
$$
SC^{\ell}_{\pit} (M, E) \,\, = \,\,   \sum_{k=0}^q 
C^{\infty}(\nu^*,  \pi^*(\wedge^k \nu^* \otimes  \End(E)))    \cap   S^{\ell-k}_{\pit} (M, E) .
$$
Set  $SC^{\infty}_{\pit}(M, E)=\bigcup_{\ell} SC^{\ell}_{\pit} (M, E)$ and 
$SC^{-\infty}_{\pit}(M, E)=\bigcap_{\ell} SC^{\ell}_{\pit} (M, E)$.
\end{definition}

In the notation  of \cite{BH2016a}, $SC^{\ell}_{\pit} (M, E)$ is a subspace of the space $SC^{0,\ell}(M,E)$ for the foliation $F_{\pit}$.  Thus the results of  \cite{BH2016a}  apply to  $SC^{\ell}_{\pit} (M, E)$.  

\begin{remark}\label{quant}
Recall the quantization map $\theta^{\alpha}$ of \cite{BH2016a}, which involves an exponential function.   If we use  $\whexp$ in defining  $\theta^{\alpha}$, and $p \in SC^{\infty}_{\pit}(M,E)$ and $u$ is a section of $\cS_{\nu} \otimes E$, then $(\theta^{\alpha}(p)(u))(x)$ depends only on $u$ restricted to $L^{\pit}_x$.  This is not true in general if we use the usual exponential function.   In particular,
$$
\theta^{\alpha}(p)(u)(x) \,\, = \,\, (2 \pi )^{-q}   \int_{\nu_x \times \nu^*_x} \hspace{-0.5cm}
e^{-i\langle X,\eta  \rangle}p(x,\eta) \alpha(x,\wexp_x(X))   \mathcal{T}^{-1}_{x,\wexp_x(X)}(u(\wexp_x(X)))\,   dX d\eta.
$$
\end{remark}

\begin{definition}
A family $p(t) \in SC^{\ell}_{\pit}(M,E)$ is an asymptotic symbol if there are symbols $p_k \in SC^{\ell-k}_{\pit}(M,E)$,  so that the following asymptotic expansion holds as $t \to 0$,
$$
p(t) \,\, \sim \,\, \sum_{k=0}^{\infty} t^{k} p_k.
$$
The leading symbol of $p(t)$ is $p_0$.
\end{definition}

Note that  $p(t) \,\, \sim \,\, \sum_{k=0}^{\infty} t^{k} p_k$ means that given any $N > 0$, 
$
\lim_{t \to 0} t^{-N}\Big( p(t) \,\, - \,\, \sum_{k=0}^{N} t^{k} p_k \Big) \,\, = \,\, 0 
$
in  $SC^{\ell - N-1}_{\pit}(M,E)$.   It does not imply that $\sum_{k=0}^{\infty} t^{k} p_k$ converges.

If the family $p(t) \in C^{\infty}( \nu^*, \pi^*(\wedge^k \nu^* \otimes\End (E))) \cap S^{\infty}_{\pit} (M, E)$, then for $t > 0$, we set 
$$
p(\eta,t)_t  \,\, = \,\, t^{k} p(t\eta,t),
$$
and  extend to all of $SC^{\infty}_{\pit} (M, E)$ by linearity.

\begin{definition}
A transverse asymptotic pseudodifferential operator ($\pit$A$\Psi$DO) is a family of operators $P_t$ on sections of $\cS_{\nu} \otimes E$ so that there is an asymptotic symbol $p(t) \in SC^{\ell}_{\pit}(M,E)$, with $P_t \, \sim \, \theta^{\alpha}(p(t)_t)$.

If $p(t) \,\, \sim \,\, \sum_{k=0}^{\infty} t^{k} p_k$, 
the leading symbol of $P_t$ is the symbol $p_0$.
\end{definition}

Recall that   $P_t \, \sim \,  Q_t$ if for  all $N \geq 0$, and all $s,k$,
$
\lim_{t \to 0} t^{-N}|| P_t  - Q_t||_{s,k} \,\, = \,\, 0,
$ 
where, $|| \cdot ||_{s,k}$ is the norm of an operator   from the usual $s$ Sobolev space associated to  $\cS_{\nu} \otimes E$  to the usual $k$ Sobolev space. 

Of course, as all these constructions are local, we may also make them for $\what{M}$, $\what{F}$ and $\what{E}$.

\section{Preliminary Results}\label{prelim}

Our calculations here are local so work equally well for both $(M,F)$ and $(\what{M},\what{F})$.  For simplicity of notation, we will work on $(M,F)$.  

For  $a \in C^{\infty}_c(\cG; \wedge^k  r^*(\nu^* ))$, set $a_t = t^k a$.  Below we will encounter expressions of the form $a \maD$, where $\maD$ is a transverse differential operator of order $\ell$ exacltly, that is it will differentiate exactly $\ell$ times, no more and no fewer.  For such operators, we set $(a \maD)_t = t^{k+\ell}a \maD$.  The reader should note that with this definition we have 
$\varsigma(\maD)_t =   t^{\ell} \varsigma(\maD)$, since $\varsigma(\maD)$ will be a polynomial in $\eta$ homogenous of degree $\ell$.

Suppose $a_0, \cdots , a_k \in C^{\infty}_c(\cG)$.  To prove Theorem \ref{CCchar}, we need  certain facts about   
$$
t^{k+2|\ell|}\big(a_0(\delta a_1)^{(\ell_1)} \cdots (\delta a_k)^{(\ell_k)}\big)_{1/t} \quad \text{    and     }  \quad \varsigma(e^{-t^2 D^2})_{1/t}.
$$
Here,  $\ell = (\ell_1,...,\ell_k) \in \N^k$,  $ \delta a_j = [ D,a_j]$,  $ \mathfrak{D}(a_j) = [ D^2,a_j]$, and $(\delta a_j)^{(\ell_j)} =  \mathfrak{D}^{\ell_j}(\delta a_j)$.     If we set 
$$
a^{k,\ell}(D) = a_0(\delta a_1)^{(\ell_1)} \cdots (\delta a_k)^{(\ell_k)},
$$
then,   
$$
t^{k+2|\ell|}a_0(\delta a_1)^{(\ell_1)} \cdots (\delta a_k)^{(\ell_k)}  \,\, = \,\,
a^{k,\ell}(tD).
$$

\medskip
\begin{Equation}
\large{\text{Results on} $\delta a = [D,a]$. } 
\end{Equation} 

First we recall some notation.  For $x \in M$, $\cG_x = s^{-1}(x)$, that is homotopy equivalence classes of paths, with end points fixed, in the leaf $L_x$ starting at $x$.  The holonomy along an element $\gamma \in \cG$ is denoted $h_{\gamma}$.   So $h_{\gamma}$ transports objects at $s(\gamma)$ to $r(\gamma)$ along the path $\gamma$ in the leaf $L_{s(\gamma)}$.   An element $a \in C^{\infty}_c(\cG; \wedge^k r^*(\nu^* ))$ acts on sections of $\cS_{\nu} \otimes E$ by
$$
a(u)(x) \,\, = \,\,  \int_{\cG_x} a(\gamma^{-1}) h_{\gamma^{-1}} u( r(\gamma))
d\gamma.
$$

\medskip
The operator $\wtit{D}$ acts on $a$ as follows.  Choose double foliation coordinate charts $U,(y,z)$ and $V,(y',z')$ respectively of $s(\gamma)$ and $r(\gamma)$, where $y$ and $y'$ are transverse coordinates and $z$ and $z'$ are tangential coordinates.  The neighborhood $(U, \gamma, V)$ of $\gamma$ consists of all leafwise paths which start in $U$, end in $V$, and are parallel with $\gamma$.  See \cite{HL} for more details.  Coordinates on $(U, \gamma, V)$  are given by $(y,z,z')$. In essence, we identify the transverse coordinates using the holonomy  $h_{\gamma}$ of $\gamma$.  The operator $\wtit{D}$ is preserved by $h_{\gamma}$, since $h_{\gamma}$ preserves the metrics on $\nu$ and $\nu^*$ and the connection $\nabla^{\nu}$.   In addition,  $\wtit{D}$ is completely determined by its restriction to the leaves of the transverse foliation.  So $\wtit{D}$ does not depend on the tangential coordinates at all, and we may assume that $\gamma$, so also $\gamma^{-1} \in (V,\gamma^{-1}, U)$, depends only on the coordinate $y$, that is we may view $a$ as a section of $\wedge^k \nu^*$ defined on the transverse leaf through $x \in U$.  Then $\wtit{D}(a)(\gamma^{-1})$ is well defined, and by definition equals $d_{\nu}a(\gamma^{-1})$. 
 
 Recall that $c(\mu)$ is Clifford multiplication by the mean curvature vector field of $F$ (thought of as a co-vector), and set $\wtit{c}(\gamma) = h_{\gamma}(c(\mu)(s(\gamma))) -  c({\mu})(r(\gamma))$. Then
\medskip
\begin{Equation}
\hspace{0.5cm}   $[tD, a]_{1/t}\,\, = \,\ d_{\nu}a + \frac{1}{2}   \wtit{c}(\mu)  a.$ 
\end{Equation}

\medskip
We have
$$
[tD, a] (u)(x) \,\, = \,\, 
t D_x \int_{\cG_x} a(\gamma^{-1}) h_{\gamma^{-1}} u(r(\gamma))d\gamma \,\, - \,\,
t\int_{\cG_x}  a(\gamma^{-1})  h_{\gamma^{-1}} ((Du)(r(\gamma)))d\gamma.
$$

Now $D = \wtit{D} -\frac{1}{2} c({\mu})$, and since $h_{\gamma^{-1}}$ preserves the metrics on $\nu$ and $\nu^*$ and the connection $\nabla^{\nu}$, we have  $[\wtit{D}, h_{\gamma^{-1}}]  = 0$.  Thus 
$$
[tD, a] (u)(x) \,\, = \,\, 
t\int_{\cG_x }d_{\nu}a(\gamma^{-1})
h_{\gamma^{-1}}u(r(\gamma))  d\gamma   \,\,  +  \,\, 
\frac{t}{2} \int_{\cG_x} a(\gamma^{-1})
( h_{\gamma^{-1}}c(\mu)(r(\gamma)) -c(\mu)(s(\gamma))  )               
h_{\gamma^{-1}}u(r(\gamma))d\gamma. 
$$

Note that 
$$
h_{\gamma^{-1}}c(\mu)(r(\gamma) ) - c(\mu)(s(\gamma)) \,\, = \,\,
h_{\gamma^{-1}}c(\mu)(s(\gamma^{-1})) -  c({\mu})(r(\gamma^{-1}))   \,\, = \,\,\wtit{c}(\mu)(\gamma^{-1}),
$$  
so
$$
 [tD, a] \,\, = \,\,  td_{\nu}a + \frac{t}{2} \wtit{c}(\mu) a ,
$$
and
$$
 [tD, a]_{1/t} \,\, = \,\ d_{\nu}a + \frac{1}{2} \wtit{c}(\mu) a,
$$
where the  normal one form parts act by Clifford multiplication. 

\medskip
\begin{Equation}
\large{\text{Results on} $t^2D^2$.} 
\end{Equation} 

Choose dual orthonormal bases, $f_1,...,f_q$ of 
$\nu^*$ and $e_1,...,e_q$ of $\nu$, which are parallel  (using $\nabla^{\nu}$) along radial lines through $x$.  From  the proof of  Proposition 3.10 of \cite{BH2016a}, we have
$$ 
D^2 \,\, = \,\,  \wtit{D} ^2   -\frac{1}{2}\wtit{D} (c(\mu))  +  \nabla_{\mu}    + \frac{1}{4} |\mu|^2, 
$$
and
$$
\wtit{D}^2 \,\, = \,\, -\sum_{j} \nabla^{\nu}_{e_j} \nabla^{\nu}_{e_j}   \,\, + \,\,  \sum_{j < k} f_j \wedge f_k \nabla^{\nu}_{[e_j,e_k]} \,\, + \,\,  \sum_{j < k} f_j \wedge f_k  \,  \Omega_E(e_j,e_k)  
\,\, + \,\,  \frac{1}{8}\kappa.
$$
The  term $\kappa$ is the scalar curvature  of the leaves of the transverse foliation $F_{\pit}$, and  $\Omega_E$ is the curvature of the bundle $E$ over $M$.   

The term $ \sum_{j < k} f_j \wedge f_k \nabla^{\nu}_{[e_j,e_k]}$ is problematic, since its symbol may have grading three, and in order to apply the results of \cite{BH2016a}, we need the terms of the symbol of $D^2$ to have  grading at most two.  This is why we assume that $\nu$ is integrable.

\begin{lemma}  If $\nu$ is integrable, $[e_j,e_k](x) = 0$,  so also $\sum_{j < k} f_j \wedge f_k \nabla^{\nu}_{[e_j,e_k]} (x) = 0$.
\end{lemma}

\begin{proof}
Denote by $P_x^{\nu}$ the plaque of the foliation determined by $\nu$ through the point $x$.  The $e_j$ are tangent to $P_x^{\nu}$, since $TP_x^{\nu} = \nu\, | P_x^{\nu}$, and parallel translation by $\nabla^{\nu}$ preserves $\nu$.  Recall that  $\nabla^{\nu}$ is the pull back of the Levi-Civita connection $\nabla^P$ on $P_x^{\nu}$ under the local projection to $P_x^{\nu}$ along the leaves of $F$.  We have 
$$
0\,\, = \,\, \Big(\nabla^{\nu}_{e_i} e_j \,\, - \,\, \nabla^{\nu}_{e_j} e_i\Big)(x) \,\, = \,\, \Big(\nabla^P_{e_i} e_j \,\, - \,\, \nabla^P_{e_j} e_i\Big)(x) \,\, = \,\,[e_i,e_j](x),
$$
since $\nabla^P$ is torsion free.  
\end{proof}
Note carefully that all the remaining terms in $\wtit{D}^2$ are completely determined by their restriction to $F_{\pit}$, and are projectable, that is invariant under the holonomy.  

Recall, \cite{BH2016a}, that 
$$
\Xi \,\, = \,\,    - \sum_{i;j \neq k} f_k \wedge f_j  \otimes e_k (\langle[e_j,X_i],X_i \rangle),
$$ 
where $X_1,...,X_p$ is a local orthonormal framing of $TF$,   is the grading two part of the Clifford form $-\wtit{D}(c(\mu))$.  Now, 
$$
t^2D^2 \,\, = \,\,     t^2   \Big( \sum_{j < k} f_j \wedge f_k  \,  \Omega_E(e_j,e_k) \,\, -  \,\, \frac{1}{2}\wtit{D} (c(\mu))        \,\, + \,\, \frac{1}{8}\kappa \,\, - \,\, \sum_{j} \nabla^{\nu}_{e_j} \nabla^{\nu}_{e_j}  \,\, + \,\,  \nabla_{\mu}    + \frac{1}{4} |\mu|^2 \Big), 
$$
which we may write as 
\begin{Equation}\label{fpf}
$t^2D^2 \,\, = \,\,  t^2   \Big( \sum_{j < k} f_j \wedge f_k  \,  \Omega_E(e_j,e_k) \,\, + \,\, \frac{1}{2} \Xi  \Big)  \,\, + \,\,
t^2 \Big(   
   \frac{1}{8}\kappa \,\, - \,\, \sum_{j} \nabla^{\nu}_{e_j} \nabla^{\nu}_{e_j}  \,\, + \,\,  \nabla_{\mu}    + \frac{1}{4} |\mu|^2  \,\, - \,\,  \frac{1}{2}(\wtit{D} (c(\mu)) \,\, + \,\,  \Xi )\Big).$
\end{Equation}
The two forms occurring in the first term are acting by Clifford multiplication.  All the operators in the second term have grading at most one, in the sense that when applied to an element of $C^{\infty}_c(\cG;r^*(\wedge \nu^*))$, e.\ g. $[tD,a]$, they raise its grading by at most one, e.\ g. 
$$
[\nabla^{\nu}_{e_j} \nabla^{\nu}_{e_j},a] \,\, = \,\, \nabla^{\nu}_{e_j} (\nabla^{\nu}_{e_j}(a)) \,\, + \,\, 2\nabla^{\nu}_{e_j}(a)\nabla^{\nu}_{e_j}.  
$$
The first term on the right has grading the same as $a$, while the second has grading one more.
In particular, 
$$
[t^2\nabla^{\nu}_{e_j} \nabla^{\nu}_{e_j}, [tD,a]]_{1/t}
$$
is polynomial in $t$ and has no constant term, so $\lim_{t \to 0} [t^2\nabla^{\nu}_{e_j} \nabla^{\nu}_{e_j}, [tD,a]]_{1/t} = 0$, and the same is true for all the other operators in the second term of Equation \ref{fpf}.  Thus, when we rescale by $1/t$, the $t^2$ in the first term will disappear, but not in the second term, so that term (when applied to an element of $C^{\infty}_c(\cG;r^*(\wedge \nu^*))$) will not play  a role in the computation of the trace.   (However, the term $\sum_{j} \nabla^{\nu}_{e_j} \nabla^{\nu}_{e_j}$, or rather its symbol, will play a crucial role in that computation.)
For simplicity we write $t^2D^2$ as
$$
t^2D^2 \,\, = \,\,  \frac{t^2}{2}( \Omega_E +  \Xi) + t^2\Lambda.
$$ 

Set 
$$
\wtit{\Xi}(\gamma)  \,\, = \,\,  \Xi(r(\gamma)) \,\, - \,\, h_{\gamma} \Xi(s(\gamma)).
$$  

\medskip
\begin{Equation}
\hspace{0.5cm}  $[t^2D^2, [tD, a]]_{1/t}   \,\, = \,\,  \Big(\frac{1}{2} \wtit{\Xi}\Big)  (d_{\nu}a + \frac{1}{2}   \wtit{c}(\mu)  a) \,\, + \,\, f_1(t),$
\end{Equation}
\noindent
where $f_1$ is polynomial in all variables, and $f_1(0) = 0$.

By the comments above about $\Lambda = \frac{1}{8}\kappa - \sum_{j} \nabla^{\nu}_{e_j} \nabla^{\nu}_{e_j}  +  \nabla_{\mu} + \frac{1}{4} |\mu|^2  -  \frac{1}{2}(\wtit{D} (c(\mu)) +  \Xi)$, we have
$$
[t^2D^2, [tD, a]]_{1/t} \,\, = \,\,  \frac{1}{2}[ \Omega_E +  \Xi,  d_{\nu}a + \frac{1}{2}   \wtit{c}(\mu)  a]   \,\, + \,\, [t^2\Lambda, [D, a]]_{1/t},
$$
where the last term is polynomial in $t$ with no constant term.   Now,  $\Omega_E$ is a basic two form, that is a two form which is constant along leaves of $F$, so $[ \Omega_E,  d_{\nu}a + \frac{1}{2}   \wtit{c}(\mu)  a]= 0$. 
The argument used in the calculation of $[tD,a]$ shows that  
$$
[t^2D^2, [tD, a]]_{1/t}   \,\, = \,\,  \Big(\frac{1}{2} \wtit{\Xi}\Big)  (d_{\nu}a + \frac{1}{2}   \wtit{c}(\mu)  a) \,\, + \,\, f_1(t),
$$
where $f_1$ is polynomial in  all variables, and $f_1(0) = 0$.

A straight forward induction argument gives
\begin{Equation}
\hspace{0.5cm} $[t^2D^2,[\cdots[t^2D^2, [tD, a]]\cdots]]_{1/t}  \,\, = \,\,  \Big(\frac{1}{2}\wtit{\Xi} \Big)^{\ell}
(d_{\nu}a + \frac{1}{2}   \wtit{c}(\mu)  a) \,\, + \,\,f_{\ell}(t),$
\end{Equation}
\noindent
where there are $\ell$ copies of $t^2D^2$, and $f_{\ell}$ is polynomial in all variables, and $f_{\ell}(0) = 0$.
Thus we have 
\begin{lemma}\label{limakl}
$$
 a^{k,\ell}(tD)_{1/t}	 \,\, = \,\,    
a_0 \Big[\prod_{j=1}^k  \Big(\frac{1}{2} \wtit{\Xi}\Big)^{\ell_j}
(d_{\nu}a_j + \frac{1}{2}   \wtit{c}(\mu)  a_j)\Big] \,\, + \,\,f_{k,\ell}(t),
$$
where $f_{k,\ell}$ is polynomial in all variables, and $f_{k,\ell}(0) = 0$.
\end{lemma}

Note that here the $d_{\nu}a_j + \frac{1}{2}   \wtit{c}(\mu)  a_j$ are now acting as differential forms, {\em not }as Clifford forms.  

\medskip
It is very important to note the following.  Each term of $f_{k,\ell}$, for $\ell \neq (0,...,0)$, is the composition of elements of 
$C^{\infty}_c(\cG, \wedge r^*(\nu^*))$ and  $\pit$A$\Psi$DOs (actually differential operators) whose leading symbols are $0$.  For example, $[t^2D^2,[t^2D^2,[tD,a]]]_{1/t}$  contains elements of 
$C^{\infty}_c(\cG, \wedge r^*(\nu^*))$  multiplied by the elements
$$
t\nabla^{\nu}_{e_j}(d_{\nu}a) \nabla^{\nu}_{e_k}\nabla^{\nu}_{e_k}\nabla^{\nu}_{e_j}  \quad \text{  and  }  \quad
t\nabla^{\nu}_{e_j}(d_{\nu}a) \nabla^{\nu}_{e_j}\nabla^{\nu}_{e_k}\nabla^{\nu}_{e_k}.
$$
Both $t\nabla^{\nu}_{e_k}\nabla^{\nu}_{e_k}\nabla^{\nu}_{e_j}$    and  
$t \nabla^{\nu}_{e_j}\nabla^{\nu}_{e_k}\nabla^{\nu}_{e_k}$ are transverse asymptotic differential operators whose leading symbols are $0$.

\medskip
\begin{Equation}\label{digress}
\large{\text{A digression}. } 
\end{Equation} 

For later use, we recall the symbol of $t^2D^2$ given in Proposition 3.10 of \cite{BH2016a}, and we assume that we are using the Dominguez metric on $M$.   As we are assuming that the normal bundle is integrable, the term $i\langle \vartheta_{\nu}, \xi \rangle = 0$ in that calculation.  Thus 
$$
{\varsigma}(D^2)(x,\eta) \,\, = \,\,    |\eta|^2    \,\, - \,\,  i\sum_j e_{j,x} e_j\langle  \wexp^{-1}_{x}(x'),\xi \rangle  \,\,  +  \,\, 
\frac{1}{8}\kappa   \,\,  +  \,\, 
\sum_{j<k} f_j \wedge f_k   \otimes \Omega_E(e_j,e_k)  \,\, + \,\, 
$$
$$
 \frac{1}{2}\sum_{i,j,k}  f_j \wedge f_k  \otimes e_k(\langle  [e_j,X_i],X_i\rangle)\,\, + \,\,   i\langle  \mu ,  \eta \rangle     \,\, - \,\,    \frac{1}{4} |\mu|^2.    
$$
Consider the term $e_{j,x} e_j\langle  \wexp^{-1}_{x}(x'),\xi \rangle$.  Note that the three other terms in the first row are projectable.  The three terms in the second row all come from $\mu$, which is projectable if we use the Dominguez  metric.  As noted in the proof of Proposition \ref{traceElem2} below, the symbol of $D^2$ is projectable if we use the Dominguez metric.  It follows that $e_{j,x} e_j\langle  \wexp^{-1}_{x}(x'),\xi \rangle$ must be projectable, and so equals $e_{j,x} e_j\langle  \wexp^{-1}_{x}(x'),\eta \rangle$, which has grading one.
Set
\begin{Equation}\label{symbol}
$$
p(x,\eta,t) \,\, = \,\,  |\eta|^2        \,\,  +  \,\,  \sum_{j<k} f_j \wedge f_k   \otimes \Omega_E(e_j,e_k)  \,\, + \,\,  \frac{1}{2}\sum_{i;j\neq k}  f_j \wedge f_k  \otimes e_k(\langle  [e_j,X_i],X_i\rangle)    \,\, + \,\,
$$
$$
t  \Big(i\langle  \mu ,  \eta \rangle   \,\,-\,\, i\sum_j e_{j,x} e_j\langle  \wexp^{-1}_{x}(x'),\eta \rangle \Big)    \,\, + \,\,    t^2   \Big( \frac{1}{8}\kappa   \,\, - \,\,    \frac{1}{4} |\mu|^2  \,\, - \,\, 
\frac{1}{2}\sum_{i, k}    e_k(\langle  [e_k,X_i],X_i\rangle) \Big).      
$$
\end{Equation} 
\noindent
Then $p(x,\eta,t)_t = \varsigma(t^2D^2)$, so  $t^2D^2$ is an $\pit$A$\Psi$DO with leading symbol 
$$
p_0^{t^2D^2}(x,\eta) \,\, = \,\,    |\eta|^2    \,\,  +  \,\, \sum_{j<k} f_j \wedge f_k   \otimes \Omega_E(e_j,e_k)  \,\, + \,\,  \frac{1}{2}\sum_{i;j\neq k}  f_i \wedge f_k  \otimes e_k(\langle  [e_i,X_j],X_j\rangle),
$$
which we write as  
$$
p_0^{t^2D^2} \,\, = \,\,    |\eta|^2    \,\,  +  \,\, \frac{1}{2}\Big(\Omega_E   \,\,  +  \,\,  \Xi \Big),
$$ 
for short. 

Note that the leading symbol of $t^2D^2 - \lambda$, so also $t^2\what{D}^2- \lambda$, is given by
$$
p_0^{t^2D^2 - \lambda} \,\, = \,\,    |\eta|^2    \,\,  +  \,\, \frac{1}{2}\Big(\Omega_E   \,\,  +  \,\,  \Xi \Big) \,\,  -  \,\, \lambda.
$$  

\medskip
\begin{Equation}
\large{\text{Results on} $\varsigma(e^{-t^2D^2})_{1/t}$. } 
\end{Equation} 

Let $e_1,...,e_q$ be a local orthonormal basis of $\nu$ with dual orthonormal basis $f_1,...,f_q$ of $\nu^*$.  Denote by $\Omega_{\nu}$ the curvature of $\nabla^{\nu}$ acting on $\cS_{\nu}\otimes E$, which is a projectable two form.  Set 
$$
\Omega_{\nu}(e_i,e_j) \,\, = \,\, \sum_{i,j,k,\ell=1}^q (\Omega_{\nu})^k_{\ell,i,j} e_k \otimes f_{\ell}, \; \text{ that is }  \;  (\Omega_{\nu})^k_{\ell,i,j} \,\, = \,\, \langle \Omega_{\nu}(e_i,e_j)(e_{\ell}), e_k \rangle,
$$
and note that $(\Omega_{\nu})^k_{\ell,i,j}$ is skew in the indices $i,j$ (since $\Omega_{\nu}$ is a 2-form) as well as the $k,\ell$,
(since $\Omega_{\nu}$ has coefficients in $so_q = spin_q)$.
Set 
$$
\Omega_{\nu}(\pa/\pa \eta,\pa/\pa \eta') \big(p(x,\eta) \wedge q(x,\eta')\big) \,\, =  
\sum_{i,j,k,\ell=1}^q(\Omega_{\nu})^k_{\ell,i,j} f_k \wedge f_{\ell} \wedge
\frac{\pa p(x,\eta)}{\pa \eta_i } \wedge \frac{\pa q(x,\eta')}{\pa \eta'_j }.
$$
Then $e^{-\frac{1}{4}\Omega_{\nu}(\pa/\pa \eta,\pa/\pa \eta') }$ is a finite sum of compositions of such operators, and the number of compositions is $\leq q/2$ because of the $f_k \wedge f_{\ell}$.   Note that  $\Omega_{\nu}(\pa/\pa \eta,\pa/\pa \eta')$ is identical to the operator used in \cite{Getz} and  \cite{BF}.  As in \cite{Getz, BF, BH2016a}, the differential operator $a_0(p,q)$, defined on pairs of elements in  $SC^{\infty}_{\pit} (M, E)$, is given by 
$$
a_0(p,q)(x, \eta) \,\, = \,\,  
e^{-\frac{1}{4}\Omega_{\nu}(\pa/\pa \eta,\pa/\pa \eta')}
p(x,\eta) \wedge q(x, \eta') \, |_{\eta'=\eta}.
$$

It is easy to check that the results of \cite{BF} extend to $SC^{\infty}_{\pit} (M, E)$ and $SC^{\infty}_{\pit} (\what{M}, \what{E})$, as do those of \cite{BH2016a}, mutatis mutandis.  For $(\what{M},\what{F})$, this is because the geometry there  is uniformly bounded.  In particular, the fact that $p_0^{t^2D^2} =     |\eta|^2   +  \frac{1}{2}\Big(\Omega_E   +    \Xi \Big)$, which is the symbol of a uniformly transversely elliptic differential operator, see \cite{Shubin}, implies that $t^2D^2$ is asymptotically elliptic, in the sense of \cite{BF}, for the symbol space $SC^{2}_{\pit} (M, E)$, and similarly for  $t^2\what{D}^2$ and $SC^{2}_{\pit} (\what{M}, \what{E})$.    Thus, combining Lemmas 3.13 and  3.14 of \cite{BF} in our context here gives
 \begin{lemma}
$\varsigma(e^{-t^2D^2})_{1/t}$ is an asymptotic symbol in $SC^{-\infty}_{\pit} (M, E)$, with leading symbol $e^{a_0(-p_0^{t^2D^2}, \cdot)} (1)$, where $1$ is the symbol which is the constant function $1$. 
\end{lemma}

Next we need to compute, for an arbitrary symbol $q \in SC^{\infty}_{\pit} (M, E)$, 
$$
a_0(-p_0^{t^2D^2},q) \,\, = \,\,  e^{-\frac{1}{4}\Omega_{\nu}(\pa/\pa \eta,\pa/\pa \eta')}
\big((-p_0^{t^2D^2})(x,\eta) \wedge q(x, \eta')\big) \, |_{\eta' = \eta}.
$$
Denote by
$\Omega_{\nu}(\pa/\pa \eta,\pa/\pa \eta)$ the operator which acts as
$$
\Omega_{\nu}(\pa/\pa \eta,\pa/\pa \eta)q(x,\eta) \,\, = \,\, 
\sum_{i,j,k,a,b,c,d=1}^q (\Omega_{\nu})^a_{b,j,i}(\Omega_{\nu})^c_{d,i,k} \,\, f_j \wedge f_k  \wedge  \frac{\pa^2 q(x,\eta)}{\pa \eta_j\pa \eta_k }.
$$
Denote by $\Omega_{\nu}(\eta,\pa/\pa \eta)$ the operator which acts as
$$
\Omega_{\nu}(\eta,\pa/\pa \eta)q(x, \eta) \,\, = \,\, 
\sum_{i,j,k,\ell=1}^q (\Omega_{\nu})^k_{\ell,i,j} f_k \wedge f_{\ell} \wedge \eta_i \frac{\pa q(x, \eta)}{\pa \eta_j }.
$$
Then, computing as in \cite{Getz} and \cite{BF},    we have 
$$
a_0(-p_0^{t^2D^2}, q)  \,\, = \,\, 
$$
$$
\Big[ - |\eta|^2     \,\,  -  \,\,  \frac{1}{2}(\Omega_E    \,\,  + \,\,     \Xi )
\,\, + \,\, \frac{1}{2}\Omega_{\nu}(\eta,\pa/\pa \eta)
\,\, + \,\,  
\frac{1}{16}\Omega_{\nu}(\pa/\pa \eta,\pa/\pa \eta)\Big]q(x, \eta).
$$
Thus,
$$
e^{a_0(-p_0^{t^2D^2}, \cdot)} (1) \,\, = \,\, e^{ \big(- |\eta|^2     \,\,  -  \,\,  \frac{1}{2}(\Omega_E   \,\,  +  \,\,     
\Xi )  \,\, + \,\, \frac{1}{2}\Omega_{\nu}(\eta,\pa/\pa \eta)  \,\, + \,\,  
\frac{1}{16}\Omega_{\nu}(\pa/\pa \eta,\pa/\pa \eta)\big)}(1).
$$
Note that 
$
\Omega_{\nu}(\eta,\pa/\pa \eta) (|\eta|^2) \,\, = \,\, 
2\sum_{i,j,k,\ell=1}^q (\Omega_{\nu})^k_{\ell,i,j} f_k \wedge f_{\ell} \otimes \eta_i  \eta_j   \,\, = \,\, 0,$
since
$(\Omega_{\nu})^k_{\ell,i,j}  = -(\Omega_{\nu})^k_{\ell,j,i}$. 
(This corrects a typo in \cite{BF}.)   Thus $\Omega_{\nu}(\eta,\pa/\pa \eta)$ and $|\eta|^2$  commute. Since 
 $\Omega_{\nu}(\eta,\pa/\pa \eta)$ also commutes with all the other operators in the exponent,  this equals 
$$
e^{ \big(- |\eta|^2     \,\,  -  \,\,  \frac{1}{2}(\Omega_E   \,\,  + \,\,     \Xi)  
\,\, + \,\,  
 \frac{1}{16}\Omega_{\nu}(\pa/\pa \eta,\pa/\pa \eta)
\big)}e^{\frac{1}{2}\Omega_{\nu}(\eta,\pa/\pa \eta)}(1).
$$
As $e^{\frac{1}{2}\Omega_{\nu}(\eta,\pa/\pa \eta)}(1) = 1$, and $\Xi$ commutes with $ |\eta|^2 +   \frac{1}{2}\Omega_E  -  \frac{1}{16}\Omega_{\nu}(\pa/\pa \eta,\pa/\pa \eta)$, (it is a transverse $2$ form, so a nilpotent multiplication operator, which involves no differentiation), we have
$$
p_0^{e^{-t^2D^2}}  \,\, = \,\, 
e^{a_0(-p_0^{t^2D^2}, \cdot)} (1) \,\, = \,\, 
e^{ -   \frac{1}{2} \Xi }  \wedge
e^{ \big(- |\eta|^2     \,\,  -  \,\,  \frac{1}{2}\Omega_E      \,\, + \,\,  
\frac{1}{16}\Omega_{\nu}(\pa/\pa \eta,\pa/\pa \eta)\big)}
(1),
$$
which we write as
$$
e^{-\frac{1}{2} \Xi }  \wedge  e^{- |\eta|^2 -\frac{1}{2}\Omega_E+ \frac{1}{16}\Omega_{\nu}}(1).
$$
Note that the second exponent contains only operators which act in normal directions, and is constant along leaves of $F$.  

Finally, the fact that $\varsigma(e^{-t^2 D^2})_{1/t}$ is an asymptotic symbol means that we have 
$$
\varsigma(e^{-t^2 D^2})_{1/t} \,\, \sim \,\,e^{-\frac{1}{2} \Xi }  \wedge e^{- |\eta|^2 -\frac{1}{2}\Omega_E+ \frac{1}{16}\Omega_{\nu}}(1) 
\,\, + \,\,  \Upsilon.
$$
Here $\Upsilon$ is a power series in $t$ with coefficients in $SC^{-\infty}_{\pit}(M,E)$, in particular,
$\dd \Upsilon  \,\, = \,\, \sum^{\infty}_{k=1} t^k p_k(x,\eta)$, where for all $N \geq 1$,
$$
\lim_{t \to 0} t^{-N} \Big( \varsigma(e^{-t^2D^2})_{1/t} - \Big[e^{-\frac{1}{2} \Xi }  \wedge e^{- |\eta|^2 -\frac{1}{2}\Omega_E+ \frac{1}{16}\Omega_{\nu}}(1) \,\, + \,\, \sum^N_{k=1}t^k p_k(x,\eta) \Big]\Big) \,\, = \,\, 0
$$
in $SC^{-\infty}_{\pit}(M,E)$.

Just as in Corollary 3.3, \cite{BF} p.\ 25, we have 
\begin{proposition}\label{intexp}
$$ \lim_{t \to 0}   \Bigl[ (2\pi)^{-q}\int_{\nu^*_x}   \Tr_s ( \varsigma(e^{-t^2D^2})_{1/t}) d\eta \Bigr]\,\, = \,\, 
(2\pi)^{-q}\int_{\nu^*_x}   \lim_{t \to 0} \Tr_s ( \varsigma(e^{-t^2D^2})_{1/t} )d\eta \,\, = \,\, 
$$
$$
 (2\pi)^{-q}  \int_{\nu^*_x}     \Tr_s( e^{-\frac{1}{2} \Xi } e^{- |\eta|^2 -\frac{1}{2}\Omega_E+ \frac{1}{16}\Omega_{\nu}}(1) )d\eta \,\, = \,\,    e^{-\frac{1}{2} \Xi }  \wedge \ch(E) \wedge \what{A}(\nu^*)](x).
$$
\end{proposition}
The last equality is classical.

\section{Proof of the Theorem \ref{CCchar}  }\label{proof}

To prove Theorem \ref{CCchar}, we need to compute the Connes-Chern character of $(C_c^\infty (\cG), \maH,D)$, and then note that the same proof works for $(C_c^\infty (\cG), B(\what{\maH})^{\Gamma}, \what{D})$, mutatis mutandis.   Combined with Proposition \ref{ASresult} below, this gives the proof of the theorem. 

Suppose $a_0, \cdots , a_k \in C^{\infty}_c(\cG)$.  In the notation of \cite{CM95}, Theorem II.3, p.\ 230, for $k >0$, $k$ even, (there is a separate formula for $\phi_0$, see below)
the Connes-Chern character of $(C_c^\infty (\cG), \maH,D)$  is given by the local formula
$$
\phi_k (a_0, \cdots , a_k)  \,\, = \,\, \sum_{\ell,q} \frac{(-1)^{|\ell|}}{\ell_1! \cdots \ell_k!} \alpha_{\ell} \sigma_q(|\ell|+ \frac{k}{2})\tau_q(\gamma a^{k,\ell}(D) (\I + D^2)^{-k/2 - |\ell|}).
$$
As above, $a^{k,\ell}(D) = a_0(\delta a_1)^{(\ell_1)} \cdots (\delta a_k)^{(\ell_k)}$, where $\delta a_j = [ D,a_j]$,  $ \mathfrak{D}(a_j) = [ D^2,a_j]$,  $(\delta a_j)^{(\ell_j)} =  \mathfrak{D}^{\ell_j}(\delta a_j)$,  $\gamma$ is the grading operator, and $\tau_q$ is a certain residue.  Only $q =0$ actually occurs, since the dimension spectrum is simple.   Note that 
$\sigma_0(k/2 +|\ell|) = \Gamma(k/2 +|\ell|)$, and  $\alpha^{-1}_{\ell} = ({\ell}_1+1)({\ell}_1 + {\ell}_2+2)\cdots({\ell}_1 + \cdots + {\ell}_k+k)$.  Then, for $k>0$,
$$
\tau_0(\gamma a^{k,\ell}(D)(\I + D^2)^{-k/2 - |\ell|}) \,\, = \,\, 
Res_{z=0} \Bigl[ \Tr_s(a^{k,\ell}(D)(\I + D^2)^{-k/2 - |{\ell}|-z})\Bigr] \,\, = \,\,
$$
$$
Res_{z=0} \Bigl[ \Tr_s(\Gamma(k/2 + |{\ell}|+ z)^{-1}  \int_0^{\infty} t^{k/2 + |{\ell}|+ z-1}a^{k,\ell}(D)e^{-t(\I+D^2)} dt ) \Bigr]   \,\, = \,\, 
$$
$$
\Gamma(k/2+ |{\ell}|)^{-1}Res_{z=0} \Bigl[ \Tr_s( \int_0^{\infty} 2t^{k + 2|{\ell}| + 2z-1} a^{k,\ell}(D)e^{-t^2(\I+D^2)}  dt ) \Bigr]  \,\, = \,\, 
$$
$$
\Gamma(k/2+ |{\ell}|)^{-1}Res_{z=0} \Bigl[ \int_0^{\infty}2 t^{2z-1} e^{-t^2}\Tr_s(a^{k,\ell}(tD)e^{-t^2 D^2})  dt \Bigr].
$$

\begin{lemma}
For any $\ep > 0$,
$$
Res_{z=0} \Bigl[ \int_0^{\infty}2 t^{2z-1} e^{-t^2}\Tr_s(a^{k,\ell}(tD)e^{-t^2 D^2})  dt \Bigr]  \,\, = \,\,
Res_{z=0} \Bigl[ \int_0^{\ep}2 t^{2z-1} e^{-t^2}\Tr_s(a^{k,\ell}(tD)e^{-t^2 D^2})  dt \Bigr].
$$
\end{lemma}

\begin{proof}
We need to show  that 
$$
\int_{\ep}^{\infty}2 t^{2z-1} e^{-t^2}\Tr_s(a^{k,\ell}(tD)e^{-t^2 D^2})  dt  \,\, = \,\,
\int_{\ep}^{\infty}2 t^{2z+k + 2|\ell|-1} e^{-t^2}\Tr_s(a^{k,\ell}(D)e^{-t^2 D^2})  dt
$$
is  finite.  

If $T$ is an operator of order $j$, the operator $[D^2,T] = |D| \, [|D|,T] + [|D|,T] \, |D|$ has order $j+1$.  So 
the operator  $a^{k,\ell}(D)$ has order $|\ell|$, and the operator $a^{k,\ell}(D) (1+D^2)^{-|\ell|/2}$ is bounded.
Now 
$$
|\Tr_s(a^{k,\ell}(D)e^{-t^2 D^2})| \,\, = \,\, |\Tr_s(a^{k,\ell}(D)(1+D^2)^{-|\ell|/2}(1+D^2)^{|\ell|/2}e^{-t^2 D^2})|
$$
and $(1+D^2)^{|\ell|/2}e^{-t^2 D^2})$ is positive and trace class.  As $|\Tr_s(AB)| \leq ||A|| \Tr(B)$ for $A$ bounded and $B$ positive and  trace class, we have
$$
|\Tr_s(a^{k,\ell}(D)e^{-t^2 D^2})| \,\, \leq \,\, ||(a^{k,\ell}(D)(1+D^2)^{-|\ell|/2}||  \Tr((1+D^2)^{|\ell|/2}e^{-t^2 D^2}).
$$
For $t \geq\ep$,
$$
\Tr((1+D^2)^{|\ell|/2}e^{-t^2 D^2}) \,\, \leq \,\, \Tr((1+D^2)^{|\ell|/2}e^{-\ep^2 D^2}).
$$
Therefore, for any $z \in \C$, 
$$
|  \int_{\ep}^{\infty}2 t^{2z+k + 2|\ell|-1} e^{-t^2}\Tr_s(a^{k,\ell}(D)e^{-t^2 D^2})  dt  |  \,\, \leq
$$
$$
||(a^{k,\ell}(D)(1+D^2)^{-|\ell|/2}|| \Tr((1+D^2)^{|\ell|/2}e^{-\ep^2 D^2}) \int_{\ep}^{\infty} 2t^{2re(z)+k + 2|\ell|-1} e^{-t^2} dt  \,\, < \,\, \infty,
$$
where $re(z)$ is the real part of $z$.
\end{proof}

Suppose that $(\what{M},\what{F})$ is a covering foliation  of $(M,F)$.
For $a \in C^{\infty}_c(\cG; \wedge^k  r^*(\nu^* ))$, set $a_t(\gamma) = t^k a(\gamma)$.  We identify  $x \in M$ with $\overline{x } \in \cG$, the class of the constant path at $x $.  So $a(x ):= a(\overline{x } )$.  Recall that the dual normal bundle of $\what{F}$ is  $\what{\nu}^*$,   and that we are thinking of $M$ as being a fundamental domain  $M \subset \what{M}$.  So for a symbol $p \in SC^{\ell}_{\pit} (\what{M}, \what{E})$, $p(x,\eta)$ makes sense for $x \in M \subset \what{M}$ and $\eta \in \what{\nu}_x^* = \nu_x^*$.
Denote by $dx_F$ the volume form along the leaves of $F$.  

The trace and supertrace  for $M$ are the usual ones, denoted $\Tr$ and $\Tr_s$, while those for $\what{M}$, which are the usual traces restricted to $M \subset \what{M}$, are denoted $\tau$ and $\tau_s$.  As noted above,  these traces are constructed using the Berezin integral, \cite{Getz, BGV92}.   This is why only the tangential volume form $dx_F$ appears in the formulas below.  

Note that the actions of $a \in C^{\infty}_c(\cG; \wedge^* r^*(\nu^*))$ are smoothing along the leaves of $F$ and $\what{F}$, which is why we may restrict our attention to operators with symbols in $SC^{\infty}_{\pit}(M,E)$ and $SC^{\infty}_{\pit}(\what{M},\what{E})$.
 
Now we invoke results of Dominguez, \cite{D98}, and Alvarez Lop\'{e}z, \cite{AL92}, which simplifiy things considerably.  Recall, \cite{CM95}, that the cohomology classes of the cocycles $(\phi_k)$  and $(\what{\phi}_k)$ do not depend on the metric on M or the induced metric on $\what{M}$. 
Dominguez has shown that it is always possible to choose the metric on  $M$ (so also on $\what{M}$) so that $c(\mu)$ is a basic form, so $\Xi$ is also, (and in fact $\Xi =d_{\nu} c(\mu) = dc(\mu)$).  Thus, if we use this metric, we have $\wtit{c}(\mu) = \wtit{\Xi} = 0$.  Alvarez Lop\'{e}z showed that the basic component of the mean curvature vector field of any Riemannian foliation is closed.   Thus  $\Xi = 0$ for the Dominguez metric.  Similarly for the corresponding operators on $(\what{M},\what{F})$.    Then Lemma \ref{limakl} immediately gives
$$
\text{if $\ell \neq (0,...,0)$,} \,\,\,    a^{k,\ell}(tD)_{1/t}	 \,\, = \,\,   \,\,f_{k,\ell}(t),  
\quad \text{  and  } \quad 
a^{k,0}(tD)_{1/t} \,\, = \,\,  a^{k,0}(t\what{D})_{1/t} \,\, = \,\, a_0 d_{\nu}a_1 \cdots d_{\nu}a_k.
$$ 

As noted above, each term of $f_{k,\ell}$, $\ell \neq (0,...,0)$, can be written as a sum of compositions of an element of  $C^{\infty}_c(\cG, \wedge^* r^*(\nu^*))$ and a $\pit$A$\Psi$DO (actually differential operator) whose leading symbol is $0$.  More specifically, denote by $e_1,...,e_q$  a local orthonormal basis of  $\nu$, and denote by $\nabla_J$ a (local) operator which is a finite composition (in any order and with repetitions allowed) of $\nabla^{\nu}_{e_1}, \cdots \nabla^{\nu}_{e_q}$ and $\nabla_{\mu}$.  Then each term of $f_{k,\ell}$ may be written locally as $t^{b_J} a_J  \nabla_J$  for some $a_J \in C^{\infty}_c(\cG, \wedge^* r^*( \nu^*))$, independent of $t$, and $b_J \geq 1$.  Similarly for $a^{k,\ell}(t\what{D})_{1/t}$ and the induced operator $ t^{b_J}a_J  \what{\nabla}_{J}$ on $\what{\maH}$.   Then 
$$
a^{k,\ell}(tD)  \,\, = \,\,  t^{k + 2|\ell |} \sum_J  a_J   \nabla_J    \,\, = \,\,   \sum_J  (a_J)_t  t^{b_J} (\nabla_J)_t,
$$  
and  
\begin{Equation}\label{akleq}
 \hspace{3.0cm}$\dd a^{k,\ell}(t\what{D})  \,\, = \,\,  t^{k + 2|\ell |} \sum_J  a_J   \what{\nabla}_J \,\, = \,\,  \sum_J  (a_J)_t   t^{b_J}(\what{\nabla}_J)_t.$
\end{Equation}
Note that both $t^{b_J}(\nabla_J)_t$   and $ t^{b_J}   (\what{\nabla}_J)_t$ are transverse asymptotic differential operators with leading symbols $0$.
 
\begin{proposition}\label{traceElem2}
For all $t > 0$,
$$
\Tr(a^{k,\ell}(tD) e^{-t^2D^2})  \,\, = \,\,   (2 \pi)^{-q}   \int_{\nu^*}
t^{k + 2|\ell | }\sum_J   \sum_{\gamma \in\cG_x^x} (a_J) _{1/t}(\gamma^{-1})   \Tr \Big(\varsigma(\nabla_J e^{-t^2D^2})_{1/t} (x, \eta) \Big) \,  d\eta dx_F,
$$
and
$$
\Tr_s(a^{k,\ell}(tD)  e^{-t^2D^2}) \,\, = \,\,   (2 \pi)^{-q}   \int_{\nu^*} 
t^{k + 2|\ell | }\sum_J  
\sum_{\gamma \in \cG_x^x} (a_J)_{1/t}(\gamma^{-1})  \Tr_s \Big(  \varsigma( \nabla_J e^{-t^2D^2})_{1/t} ( x, \eta) \Big)\,  d\eta dx_F.
$$

The corresponding results hold for $\tau(a^{k,\ell}(t\what{D}) e^{-t^2 \what{D}^2})$  and $\tau_s(a^{k,\ell}(t\what{D})e^{-t^2\what{D}^2}) $, where $\nu^*$ is replaced by $\what{\nu}^*|M \subset \what{M}$, and $\sum_{\gamma \in \cG_x^x}(a_J)_{1/t}(\gamma^{-1}) $s replaced by $\sum_{\what{\gamma} \in \what{\cG}_x^x} (a_J)_{1/t}((\rho \circ \what{\gamma})^{-1}) $.
\end{proposition}%

\begin{proof}     
We only do the proof for  $\Tr_s(a^{k,\ell}(tD)  e^{-t^2D^2})$,  as the other three cases are quit similar.

First we prove it  without applying the rescaling operator $(\cdot)_{1/t}$.  
From \cite{BH2016b} we have
\begin{lemma}\label{traceclass}
For all $t > 0$,  $a_J \nabla_Je^{-t^2D^2}$ is a smoothing operator on $\maH$.  Similarly for  $a_J \what{\nabla}_Je^{-t^2\what{D}^2}$  acting on $\what{\maH}$.
\end{lemma}
\noindent

Thus we have 
$$
\Tr_s(a_J {\nabla}_J e^{-t^2{D}^2}) \,\, = \,\, \int_M \Tr_s(K(x,x)) \, dx_F,$$
where $K$ is the Schwartz kernel of $a_J {\nabla}_J e^{-t^2{D}^2}$, and we using the Berezin integral for spinors,  \cite{Getz, BGV92}.  

Denote by $\delta_x$, $\delta_0$, $\delta^{{F}}_0$, and $\delta^{{L}}_x$ the Dirac delta operators at $x \in M$, $0 \in T^*{M}_x$, $0 \in T^*{F}_x$, and  $x \in {L}_x$, respectively.  For ${u}_{x} \in ({\cS}_{\nu} \otimes {E})_x $,
$$
K(x,x){u}_{x}       \,\, = \,\,
\Big(a_J {\nabla}_Je^{-t^2{D}^2}\big(\delta_x ({x'}) {\mathcal{T}}_{x,{x'}}({u}_{x})\big) \Big)(x)   \,\, = \,\,
\Big(a_J {\nabla}_Je^{-t^2{D}^2}\big(\delta_0({\what{\exp}}_{x}^{-1}({x'})){\alpha}(x, {x'}){\mathcal{T}}_{x,{x'}} ({u}_{x})\big)\Big)(x).
$$
Now the Fourier Transform of $\delta_0$ is the constant function $1$.  So 
$$
\Tr_s(K(x,x)) \,\, = \,\,    \Tr_s \Big( \big( a_J {\nabla}_Je^{-t^2{D}^2} \big[ (2 \pi)^{-n}   \int_{T^*{M}_x} e^{i\langle {\what{\exp}}^{-1}_{x}({x'}),\xi \rangle}{\alpha}(x, {x'}){\mathcal{T}}_{x,{x'}}  \, d\xi  \big] \big) (x) \Big)   \,\, = \,\,    
$$
$$
(2 \pi)^{-q}   \int_{{\nu}^*_x} \Tr_s \Big( \big(a_J  {\nabla}_J e^{-t^2{D}^2}  \big[ (2 \pi)^{-p}   \int_{T^*{F}_x}e^{i\langle {\what{\exp}}^{-1}_{x}({x'}), (\zeta,0)\rangle }e^{i\langle {\what{\exp}}^{-1}_{x}({x'}), (0,\eta) \rangle}
{\alpha}(x, {x'})  {\mathcal{T}}_{x,{x'}} d\zeta \big]  \big)    (x)  \,   \Big)   d\eta,  
$$
where we are able to move the integration outside by the argument in the proof of Theorem 3.7 of \cite{Getz}.
Since the symbol of ${\nabla}_Je^{-t^2{D}^2}$ depends only on $\eta$, we may write this as
$$
(2 \pi)^{-q}   \int_{{\nu}^*_x} \Tr_s \Big( \big( a_J  \big[ (2 \pi)^{-p}   \int_{T^*{F}_x}e^{i\langle {\what{\exp}}^{-1}_{x}({x'}), (\zeta,0)\rangle}  d\zeta \, {\nabla}_J  e^{-t^2{D}^2}( e^{i\langle {\what{\exp}}^{-1}_{x}({x'}), (0,\eta) \rangle}
{\alpha}(x, {x'})  {\mathcal{T}}_{x,{x'}} ) \big]  \big)(x)  \Big)  \,  d\eta   \,\, = \,\,  
$$
\begin{multline*}
(2 \pi)^{-q}  \int_{{\nu}^*_x} \Tr_s 
\Big(  \int_{{\cG}_x}  \, a_J  (\gamma^{-1})\\
 h_{{\gamma}^{-1}} \Big[\Big( (2 \pi)^{-p} \int_{T^*{F}_x}  \hspace{-0.2cm}e^{i\langle {\what{\exp}}^{-1}_{x}(x'), (\zeta,0)\rangle} \, d\zeta \,\, {\nabla}_J  e^{-t^2{D}^2}  \big( e^{i\langle {\what{\exp}}^{-1}_{x}x'), (0,\eta) \rangle}{\alpha}(x, x')  {\mathcal{T}}_{x,x'}\big) \Big) ({r}({\gamma}) )  \Big] d{\gamma} \Big)   \, d\eta.
\end{multline*}
Now, at ${\gamma} \in {\cG}_x$, 
$$
\Big( (2 \pi)^{-p}\int_{T^*{F}_x} \hspace{-0.2cm}     e^{i\langle {\what{\exp}}^{-1}_{x}(x'), (\zeta,0)\rangle}   \, d\zeta \Big) ({r}({\gamma} ))  \,\, = \,\,  
\delta^{{F}}_0  (({\what{\exp}}_{x} \, | _{{L}_x})^{-1}({r}({\gamma} )))  {\alpha}(x,{r}({\gamma} ))  \,\, = \,\,
\delta^{{L}}_x({r}({\gamma} )).
$$
Thus, 
\begin{multline*}
\Tr_s(K(x,x))    \,\, = \,\,    
(2 \pi)^{-q} \int_{{\nu}^*_x}     \Tr_s 
\Big(  \int_{{{\cG}_x}}\delta^{{L}}_x({r}(\gamma ))\,  a_J(\gamma^{-1})   \\
          h_{{\gamma}^{-1}}  \Big[{\nabla}_J e^{-t^2{D}^2}\Big( ( e^{i\langle {\what{\exp}}^{-1}_{x}({x'}), (0,\eta) \rangle}{\alpha}(x, {x'}))  {\mathcal{T}}_{x,{x'}} ) \, |_{{L}_x} \Big)({r}({\gamma}) )  \Big]d{\gamma} \Big) \, d\eta    \,\, = \,\, 
\end{multline*}
$$
(2 \pi)^{-q}   \int_{ {\nu}^*_x}     \sum_{{\gamma} \in {\cG}_x^x}  \Tr_s \Big( a_J(\gamma^{-1}) h_{{\gamma}^{-1}}  \Big[ \varsigma({\nabla}_Je^{-t^2{D}^2})( x, \eta) \Big] \Big) \, d\eta\,\, = \,\,  
$$
$$ 
(2 \pi)^{-q}   \int_{ {\nu}^*_x}  \sum_{{\gamma} \in {\cG}_x^x} a_J(\gamma^{-1})  \Tr_s \Big(\varsigma( {\nabla}_J e^{-t^2{D}^2})(x, \eta) ) \Big) \, d\eta,
$$
since $r(\gamma)=x$  for ${\gamma} \in {\cG}_x^x$, and 
$h_{{\gamma}^{-1}}\Big[ \varsigma({\nabla}_J e^{-t^2{D}^2})( x, \eta) \Big] = \varsigma({\nabla}_Je^{-t^2{D}^2})( x, \eta)$, as this symbol is invariant under the holonomy action.
We thus have
$$
\Tr_s(a_J {\nabla}_J e^{-t^2{D}^2})  \,\, = \,\,   (2 \pi)^{-q}   \int_{\nu^*}  \sum_{{\gamma} \in{\cG}_x^x} a_J(\gamma^{-1}) \Tr_s \Big( 
 \varsigma({\nabla}_Je^{-t^2{D}^2})(x, \eta) \Big) \,  d\eta dx_F.
$$

When we apply the rescaling operator $(\cdot)_{1/t}$, the integral is unchanged  just as in \cite{Getz}  p.\ 175.  Replacing $\eta $ by $\eta/t$ gives a factor of $t^{q}$ when doing the integration in $\eta$.  The only part of $ (a_J)_{1/t}(x)\varsigma(\nabla_Je^{-t^2D^2})_{1/t}$ which contributes comes from $C^{\infty}(\wedge^q \nu^*) \otimes_{C^{\infty}(M)}    C^{\infty}  (T^*M, \End(E))$, since we are using the Berezin integral, and this contributes the factor $t^{-q}$, which cancels the $t^q$.

The result now follows by summing over $J$ and multiplying by $t^{k + 2|\ell |}$.
\end{proof}

Now suppose that $\ell \neq (0,...,0)$.   Then
$$
 \int_0^{\ep}2 t^{2z-1} e^{-t^2}\Tr_s(a^{k,\ell}(tD)e^{-t^2 D^2})  dt \,\, = \,\,
 $$
 $$
\int_0^{\ep}2 t^{2z-1} e^{-t^2} (2 \pi)^{-q} \int_{\nu^*} 
t^{k + 2|\ell | }\sum_J  \sum_{\gamma \in  \cG_x^x} (a_J)_{1/t}(\gamma^{-1}) 
 \Tr_s \Big(  \varsigma( \nabla_J e^{-t^2\what{D}^2})( x, \eta)_{1/t} \Big)\,  d\eta dx_F dt.
$$

Set 
$$
 f(t) = e^{-t^2} (2 \pi)^{-q} \int_{\nu^*} 
t^{k + 2|\ell | }\sum_J  \sum_{\gamma \in  \cG_x^x} (a_J)_{1/t}(\gamma^{-1}) 
 \Tr_s \Big(  \varsigma( \nabla_J e^{-t^2\what{D}^2})( x, \eta)_{1/t} \Big)\,  d\eta dx_F.
 $$
The results in Section \ref{prelim} imply that $f(t)$ is continuous at $t=0$.  So   
$$
 \tau_0(\gamma a^{k,\ell}(D)(\I+D)^{-k/2 - |\ell|}) \,\, = \,\, 
\Gamma(k/2+|\ell|)^{-1}Res_{z=0} \Bigl[  \int_0^{\ep} 2 t^{2z-1}f(t) dt \Bigr],
$$  
which is approximated by  (and, because it is independent of $\ep$, actually equals)
$$
\Gamma(k/2+|\ell|)^{-1}Res_{z=0} \Bigl[ \int_0^{\ep} 2 t^{2z-1}f(0) dt \Bigr]  \,\, = \,\,  \Gamma(k/2+|\ell|)^{-1}Res_{z=0} \Bigl[  \frac{t^{2z}}{z}  f(0)  \,\, |_0^{\ep} \Bigr] \,\, = \,\,
$$
$$
\Gamma(k/2+|\ell| )^{-1} Res_{z=0} \Bigl[\frac{\ep^{2z}}{z} f(0)  \Bigr]     \,\, = \,\,
\Gamma(k/2+|\ell| )^{-1} \lim_{t \to 0}f(t)      \,\, = \,\,    
$$
$$ 
\Gamma(k/2+|\ell|)^{-1} \lim_{t \to 0}  \Bigl[e^{-t^2} (2 \pi)^{-q} \int_{\nu^*} 
t^{k + 2|\ell | }\sum_J  \sum_{\gamma \in  \cG_x^x} (a_J)_{1/t}(\gamma^{-1}) 
 \Tr_s \Big(  \varsigma( \nabla_J e^{-t^2\what{D}^2})( x, \eta)_{1/t} \Big)\,  d\eta dx_F  \Bigr].  
$$
By the results in Section \ref{prelim}, we may interchange the limit and the integration to get
$$ 
\Gamma(k/2+|\ell|)^{-1} (2 \pi)^{-q} \int_{\nu^*}   \lim_{t \to 0}  \Bigl[e^{-t^2} 
t^{k + 2|\ell | }\sum_J  \sum_{\gamma \in  \cG_x^x} (a_J)_{1/t}(\gamma^{-1}) 
 \Tr_s \Big(  \varsigma( \nabla_J e^{-t^2\what{D}^2})( x, \eta)_{1/t} \Big)\Bigr]\,  d\eta dx_F .  
$$
By Theorem 6.1 of \cite{BH2016a}, Equation \ref{akleq},  and the results of Section \ref{prelim},
$$
 \lim_{t \to 0}  
t^{k + 2|\ell | }\sum_J (a_J)_{1/t} \varsigma( \nabla_J e^{-t^2\what{D}^2})( x, \eta)_{1/t}  \,\, = \,\, 
 \lim_{t \to 0}  
 \sum_J a_Je^{-\frac{1}{4}\Omega_{\nu}(\pa/\pa\eta,\pa/\pa \eta') }  p^J_0(x, \eta) \wedge q_0(x, \eta') \, |_{\eta' =  \eta},
$$
where  $p^J_0(x, \eta)$ is the leading symbol of $t^{b_J}\nabla_J$ and $q_0(x, \eta')$ is the leading symbol of $e^{-t^2\what{D}^2}$, with $\eta'$ substituted for $\eta$.   Since $b_J \geq 1$,  $p^J_0(x, \eta) = 0$, and 
$$
 \lim_{t \to 0}  
t^{k + 2|\ell | }\sum_J (a_J)_{1/t} \varsigma( \nabla_J e^{-t^2\what{D}^2})( x, \eta)_{1/t}  \,\, = \,\, 0.
$$
Thus, if $\ell \neq (0,...0)$,
$$
\tau_0(\gamma a^{k,\ell}(D)(\I + D^2)^{-k/2 - |\ell|}) \,\, = \,\,  0.
$$

\medskip
For the case $k \neq 0$, and $\ell = (0,...,0)$, we have $t^k \sum_J (a_J)_{1/t} =  a^{k,0}(tD)_{1/t} = a_0 d_{\nu}a_1 \cdots  d_{\nu}a_k$.  The results in Section \ref{prelim} give
$$
\lim_{t \to 0} 
\Tr_s \Big(  \varsigma( e^{-t^2\what{D}^2})_{1/t} \Big)  \,\, = \,\, e^{-\frac{1}{2} \Xi }  \wedge  e^{- |\eta|^2 -\frac{1}{2}\Omega_E+ \frac{1}{16}\Omega_{\nu}}(1)  \,\, = \,\,   e^{- |\eta|^2 -\frac{1}{2}\Omega_E+ \frac{1}{16}\Omega_{\nu}}(1),
$$
since $ \Xi  = 0$ for the Dominguez metric.  Proposition \ref{intexp} then finishes the proof for this case. 

\medskip
For the case $k=0$ we have, following \cite{CM95}, p.\ 230, 
$$
\phi_0(a_0) \,\, = \,\, Res_{z=0} \Bigl[ \Tr_s(z^{-1} a_0(\I + D^2)^{-z})  \Bigr] \,\, = \,\, 
Res_{z=0} \Bigl[ \Tr_s((z\Gamma( z))^{-1}  a_0\int_0^{\infty} t^{ z-1}e^{-t(\I+D^2)} dt ) \Bigr]   \,\, = \,\, 
$$
$$
Res_{z=0} \Bigl[ \Tr_s(  a_0\int_0^{\infty} t^{ z-1}e^{-t(\I+D^2)} dt ) \Bigr],
$$
since $z\Gamma( z) = \Gamma( z+1)$, so $z\Gamma( z) \, | _{z=0} = 1$.
Then proceed just as in the case $k > 0$ to get 
$$
\quad  \phi_{0} (a_0) \,\, = \,\, 
\int_{M}  \sum_{\cG_x^x}  a_0 (\gamma^{-1}) 
 \ch(E) \wedge {\what A} (\nu^*) \wedge  dx_F.
$$

To complete the proof of Theorem \ref{CCchar}, we have the following. 
\begin{proposition}\label{ASresult}
Denote the cohomology classes of  the sequences $(\phi_{k})_{k\geq 0}$, and $(\what{\phi}_{k})_{k\geq 0}$, $k$ even, by  $\ch(M,F)$ and $\ch(\what{M}, \what{F})$ respectively.  Then the following diagram commutes on the image of the Baum-Connes assembly map in $K_*(C^*_{\max}(\cG))$.  

\begin{picture}(400,70)
\put(150,55){$K_*(C^*_{\max}(\cG))$}
\put(225,63){$\ch(M,F)$}
\put(215,59){\vector(1,0){60}}
\put(195,45){\vector(1,-1){35}}
\put(165,20){$\ch(\what{M}, \what{F})$}
\put(285,50){{\vector(-1,-1){38}}}
\put(285,55){$\Z$}
\put(235,5){$\R$}
\end{picture}

\end{proposition}  

\begin{proof}   
If $X$ is in the image of the Baum-Connes assembly map, it can be represented by  $[e] - [e']$, where $e$ and $e'$  are projections in $M_\infty( \wtit{C}^{\infty}_c(\cG))$ with the supports of the corresponding elements of $M_\infty( {C}^{\infty}_c(\cG))$ in an arbitrarily small neighborhood of the units $\cG^{(0)} \sim M \subset \cG$. See Remark \ref{BC-Free}.  The tilde indicates that a unit has been added.  Then the supports of the elements $e (d_{\nu}e)^k$ and $e' (d_{\nu}e')^k$ are also contained in an arbitrarily small neighborhood of $\cG_0 \sim M \subset \cG$. 

Apply the first part of Theorem  \ref{CCchar} to get 
$$
\langle \ch(M,F), X \rangle \,\, = \,\,    \frac{1}{k!} \int_{M}  \sum_{\gamma \in \cG_x^x} \big[ \tr(e (d_{\nu}e)^k) - \tr(e' (d_{\nu}e')^k)\big](\gamma^{-1})  \wedge \ch(E) \wedge \what{A} (\nu^*) \wedge  dx_F,
$$
and
$$
\langle \ch(\what{M}, \what{F}), X \rangle  \,\, = \,\,    \frac{1}{k!} \int_{M \subset \what{M}} 
 \sum_{\what{\gamma} \in \what{\cG}_x^x} \big[ \tr(e (d_{\nu}e)^k) - \tr(e' (d_{\nu}e')^k)\big]((\rho \circ\what{\gamma})^{-1})    
 \wedge \ch(\what{E}) \wedge {\what A} (\what{\nu}^*) \wedge  dx_F.
$$
Because of the restrictions on the supports, both sums collapse to just the constant path at $x$, so the two expressions are equal.
\end{proof}

\section{Statement of the Theorem for transverse spectral triples} \label{statementCM}

In this section we will be content to give the statement of the theorem for two transverse spectral triples associated to $F$.  The proof of Theorem \ref{CCchar} can be easily extended, using the Getzler calculus \cite{Getz}, to prove Theorem \ref{tranversethm} below.  The advantage of using these spectral triples is that we do not have to assume that the normal bundle of $F$ is integrable.

The  transverse spectral triples are given as follows.  

\medskip
1.  The Type I Connes-Moscovici spectral triple: $(C^{\infty}_c(\cG^T_T), \maH_T, D_T)$.   $T$ is a complete transversal for the foliation $F$,  and the trace is the usual super trace, constructed using the Berezin integral, for operators on $\maH_T = L^2(T, (\cS_{\nu} \otimes E) |_T)$.  $\cG^T_T \subset \cG$ is the sub groupoid of elements which start and end in $T$.  $D_T$ can be taken to be either $D$ or $\wtit{D}$, restricted to $T$.  

This transverse spectral triple is Morita equivalent to the Type I spectral triple given is Section \ref{statement}, and so is an even spectral triple with simple dimension spectrum contained in the set  $\{k \in \N \, | \, k \leq q \}$.  See \cite{K97}.

\medskip
2.  The Type II Atiyah spectral triple: $(C^{\infty}_c(\cG^T_T),  B(\what{\maH}_{\what{T}})^{\Gamma},  \what{D}_{\what{T}})$.  
$\rho:(\what{M}, \what{F})\to (M,F)$ is a (possibly non-compact) Galois foliation cover of $M$, with covering group $\Gamma$.   
The transversal $\what{T}$ is the inverse image of $T$.  $\what{D}_{\what{T}}$  and $\what{\cS}_{\nu} \otimes\what{E}$ are the pull backs of  $D_T$ and $\cS_{\nu} \otimes E$.    $\what{\maH}_{\what{T}}  =L^2(\what{T}, (\what{\cS}_{\nu} \otimes \what{E}) |_{\what{T}})$.  $B(\what{\maH}_{\what{T}})^{\Gamma}$ is the algebra of  bounded operators on $\what{\maH}_{\what{T}}$ which are $\Gamma$-invariant.
The trace is the usual super trace, constructed using the Berezin integral, restricted to a fundamental domain of $T$ in $\what{T}$, denoted $T \subset \what{T}$.  

As in the main result:  the algebra is the same as in the Type I case; $B(\what{\maH}_{\what{T}})^{\Gamma}$ is the analog of the von Neuman algebra used by  Atiyah in the $L^2$ covering index theorem; $\what{D}_{\what{T}}$ is $\Gamma$-invariant;  for a non-compact covering foliation,  this is Morita equivalent to the Type II spectral triple given is Section \ref{statement}, and so is an even spectral triple with simple dimension spectrum contained in the set  $\{k \in \N \, | \, k \leq q \}$; for a compact covering, this is a slight modification of the Type I Connes-Moscovici spectral triple above.

\begin{theorem}\label{tranversethm}
Let $\rho:(\what{M}, \what{F})\to (M,F)$ be a Galois covering foliation of $(M,F)$.  Assume that  $F$  has even codimension and is transversely spin.   Let $a_0, \ldots, a_k \in C_c^\infty (\cG^T_T)$.  The  sequences $(\phi_{k})_{k\geq 0}$, and $(\what{\phi}_{k})_{k\geq 0}$, $k$ even, where
$$
\phi_{k} (a_0, \ldots, a_k) \,\, = \,\,    \frac{1}{k!} \int_{T}  \sum_{\gamma \in \cG_x^x} (a_0  d_{\nu} a_1  \cdots d_{\nu} a_k)(\gamma^{-1})  \wedge \ch(E) \wedge \what{A} (\nu^*),
$$
and
$$
\what{\phi}_{k} (a_0, \ldots, a_k) \,\, = \,\,    \frac{1}{k!} \int_{T \subset \what{T}} 
 \sum_{\what{\gamma} \in \what{\cG}_x^x} (a_0  d_{\nu} a_1  \cdots d_{\nu} a_k)((\rho \circ\what{\gamma})^{-1})    
 \wedge \ch(\what{E}) \wedge {\what A} (\what{\nu}^*)
$$
are cocycles in the $(b,B)$ bicomplex for $C_c^\infty (\cG^T_T)$.

The cohomology classes of  $(\phi_k)$ and $(\what{\phi}_k)$ are the Connes-Chern characters of  $(C^{\infty}_c(\cG^T_T), \maH_T, D_T)$ and  $(C^{\infty}_c(\cG^T_T), B(\what{\maH}_{\what{T}})^{\Gamma},  \what{D}_{\what{T}})$, respectively.  These classes induce the same map on the image of the Baum-Connes assembly map in $K_*(C^*_{\max}(\cG^T_T))$, \cite{BC00}.   That is,  we have an Atiyah $L^2$ covering index theorem for these spectral triples associated to foliation coverings.
\end{theorem}

Finally, we point out that the above Morita reduction method of restriction to a complete transversal is not generally accessible when the action of $\Gamma$ is   not free.

\begin{example}\label{Notrans}

Recall the spectral triple 
$$
(C^{\infty}_c(\cG) \rtimes \Gamma, \mathcal{N} \subset B(M,E \otimes
\wedge \nu^*) \otimes l^2\Gamma, D_{E}\rtimes \Gamma)
$$
from \cite{BH2016b}.  Here $\Gamma$ is a countable group of diffeomorphisms acting properly,  {\underline{but not freely}},  on the possibly non-compact manifold $M$,  preserving a foliation $F$ with normal bundle $\nu^*$ and homotopy graph $\cG$.  $M/\Gamma$ is assumed to be a compact space. The bundle $E$ is basic, $\Gamma$ equivariant, and Hermitian.   The operator $D_E$ is a twisted transverse Dirac operator for $F$, and   $\mathcal{N}$ is a certain von Neumann algebra.   Given $\phi \in C^{\infty}_c(\what{\cG}) \rtimes \Gamma$,   $\phi(g) \in C^{\infty}_c(\what{\cG})$, $g \in \Gamma$, also acts on sections of $E \otimes \wedge^* \nu^*$.  Finally,  the trace $\TR$ used is:  for certain  $A \in \maN$, and all $\phi \in C^{\infty}_c(\what{\cG}) \rtimes \Gamma$, $\phi \circ A$ is $\TR$ trace class with
 $$
 \TR(\phi\circ A) \,\, = \,\,  \sum_{g \in \Gamma} \Tr(g^{-1} \phi(g) \circ A_{g,e}),
 $$
 where $ \Tr$ is just the usual trace.
 
\end{example}

As $\Gamma$ does not act freely, it is not possible in general to reduce the associated index problem  to the corresponding one on a complete transversal. Indeed, a $\Gamma$-equivariant complete transversal does not  always exist  \cite{BH2016b}, so the Morita reduction method to a transversal is not available, and we must use our global constructions.


\begin{thebibliography}{C94/95}

\bibitem[AL92]{AL92} J. A. Alvarez L\'{o}pez
{\em The basic component of the mean curvature of Riemannian foliations}, 
Ann. Global Anal. Geom. {\bf 10} (1992) 179--194.

\bibitem[At76]{AtiyahCovering} M. {F}. Atiyah,
{\em Elliptic operators, discrete groups and von Neumann algebras},  
Colloque ``Analyse et Topologie'' en l'Honneur de Henri Cartan (Orsay, 1974),  Asterisque, {\bf 32-33}, Soc. Math. France, Paris, (1976) 43--72. 

\bibitem[ABP73]{ABP} M. {F}. Atiyah, {R}. Bott, and V. K. Patodi, 
{\em On the heat equation and the index theorem}, 
Invent. Math. {\bf 19} (1973) 279--330.

\bibitem[BC00]{BC00} P. Baum and A. Connes, {\em  Geometric K-theory for Lie groups and foliations}, Enseign. Math. {\bf (2) 46} (2000) no. 1-2, 3--42.

\bibitem[B03]{B03}  M-T. Benameur, {\em Noncommutative geometry and abstract integration theory},  Geometric and topological methods for quantum field theory (Villa de Leyva, 2001), 157--227, World Sci. Publ., River Edge, NJ, 2003.

\bibitem[BF06]{BF06}  M-T. Benameur and T. Fack,  {\em Type II non-commutative geometry. I. Dixmier trace in von Neumann algebras},  Adv. Math. {\bf 199} (2006)  29--87.

\bibitem[BH17a]{BH2016a} M-T. Benameur and J. L. Heitsch,  {\em A symbol calculus for  foliations},  J. Noncommut. Geo.  {\bf 11} (2017) 1141--1194.

\bibitem[BH17b]{BH2016b} M-T. Benameur and J. L. Heitsch,  {\em Transverse noncommutative geometry of foliations}, {{arXiv:1804.06837}}

\bibitem[BGV92]{BGV92}
N. Berline, E. Getzler, and M. Vergne,  {\em Heat Kernels and Dirac Operators},
Springer-Verlag, Berlin-New York, 1992.

\bibitem[BlF90]{BF} J. Block and J. Fox, 
{\em Asymptotic pseudodifferential operators and index theory}, 
Contemporary Math. {\bf 105} (1990) 1--32.

\bibitem[C87]{ConnesNCG} A. Connes,  {\em Noncommutative geometry. Nonperturbative quantum field theory} (Cargse, 1987), 33-69, NATO Adv. Sci. Inst. Ser. B Phys., 185, Plenum, New York, 1988.

\bibitem[C94]{ConnesBook} A. Connes,
\newblock {\em Noncommutative Geometry}, Academic Press, New York,
1994.

\bibitem[CM95]{CM95} A. Connes and H. Moscovici,  
{\em The local index formula in noncommutative geometry},   
Geom. Funct. Anal. {\bf 5} (1995)  174--243.

\bibitem[CM98]{CM98} A. Connes and H. Moscovici,  
{\em Hopf algebras, cyclic cohomology and the transverse index theorem},   
Comm. Math. Phys.  {\bf 198} (1998)  199--246.

\bibitem[D98]{D98} D. Dominguez, {\em Finiteness and tenseness theorems for Riemannian foliations}, 
Amer. J. Math. {\bf 120} (1998) 1237--1276.

\bibitem[G83]{Getz} E. Getzler, 
{\em Pseudodifferential operators on supermanifolds and the Atiyah-Singer Index Theorem}, Commun. Math. Phys.  {\bf 92} (1983) 163--178.

\bibitem[GlK91]{GlK91} J. F. Glazebrook  and F. W. Kamber, 
{\em Transversal Dirac families in Riemannian foliations},  Comm. Math. Phys. {\bf 140} (1991) 217--240.

\bibitem[HL90]{HL}J.~L. Heitsch and C.~Lazarov,
{\em A {L}efschetz theorem for foliated manifolds},
Topology {\bf 29} (1990) 127--162.
 
\bibitem[K97]{K97} Yu. Kordyukov, 
{\em Noncommutative spectral geometry of Riemannian foliations}, Manuscripta Math. {\bf 94} (1997) 45--73.

\bibitem[K07]{K07} Yu. Kordyukov, 
{\em The Egorov theorem for transverse Dirac type operators on foliated manifolds},  J. Geom. Phys. {\bf 57} (2007) 2345-2364.

\bibitem[Sh92]{Shubin}  M. A. Shubin, {\em Spectral theory of elliptic operators on non-compact manifolds}, Ast\'erisque No. 207 (1992), 5, 35-108. 

\bibitem[T97]{T97} P.  Tondeur, {\em Geometry of foliations},
Monographs in Mathematics, 90.  Birkhauser Verlag, Basel, 1997.

\bibitem[Tu99]{TuHyper} J.-L. Tu,  {\em La conjecture de Novikov pour les feuilletages hyperboliques}, 
K-Theory {\bf 16} (1999) 129--184.

\bibitem[W80]{Wid2} H. Widom, {\em A complete symbolic calculus for pseudodifferential operators}, Bull. Sc. Math., 2nd Series {\bf 104} (1980) 19-63.  



\end{thebibliography}
\end{document}